  \documentclass[letterpaper,oneside]{article}

  \usepackage{amssymb,amsmath,amsthm,epsfig,mathrsfs}
  \usepackage{geometry} 
  \usepackage{fourier}
  \usepackage{xspace, xcolor} 

  \usepackage[
     breaklinks,colorlinks,
     citecolor=teal,linkcolor=magenta,urlcolor=teal
  ]{hyperref}

  \newcommand\address[1]{}
  \newcommand\email[1]{}
  \newcommand\dedicatory[1]{}

  \newcommand{\balpha}{\overline{\alpha}}
  \newcommand{\bbeta}{\overline{\beta}}

  \newcommand{\calB}{{\mathcal B}}

  \newcommand{\calE}{{\mathcal E}}

  \newcommand{\calM}{{\mathcal M}}

  \newcommand{\calX}{{\mathcal X}}

  \newcommand{\s}{\ensuremath{S_g}\xspace}                       
  \renewcommand{\sp}{\ensuremath{S_{g,p}}\xspace}                

  \newcommand{\Teich}{Teichm\"uller\xspace}                      
  \newcommand{\M}{\ensuremath{\calM_g}\xspace}                   
  \newcommand{\PM}{\ensuremath{\calM_{g,p}}\xspace}              
  \newcommand{\Mp}{\ensuremath{\PM / \Sym_p}\xspace}             

  \newcommand{\Mthick}{\ensuremath{\M^{\,\ep}}\xspace}           
  \newcommand{\PMthick}{\ensuremath{\calM_{g,p}^{\,\ep}}\xspace} 
  \newcommand{\Mpthick}{\ensuremath{\PMthick/ \Sym_p}\xspace}    
  \newcommand{\Bers}{\ensuremath{{\calB}_g}\xspace}              
  \newcommand{\PBers}{\ensuremath{{\calB}_{g,p}}\xspace}         
  \newcommand{\Bersp}{\ensuremath{\PBers / \Sym_p}\xspace}         

  \newcommand{\dT}{\ensuremath{d_T\xspace}}                      
  \newcommand{\dL}{\ensuremath{d_L\xspace}}                      
  \newcommand{\diamT}{\ensuremath{\diam_T}\xspace}               

  \newcommand{\Tree}{\ensuremath{{\rm Tree}(n)}\xspace}          
  \newcommand{\Graph}{\ensuremath{{\rm Graph}(g)}\xspace}        
  \newcommand{\PGraph}{\ensuremath{{\rm Graph}(g,p)}\xspace}     
  \newcommand{\Graphp}{\ensuremath{\PGraph / \Sym_p}\xspace}      

  \newcommand{\X}{\ensuremath{\mathcal{X}_n}\xspace}             
  \newcommand{\Xthick}{\ensuremath{\X^\epsilon}\xspace}          
  \newcommand{\F}{\ensuremath{\mathbb{F}_n}\xspace}              

  \DeclareMathOperator{\diam}{diam}                              
  \DeclareMathOperator{\ch}{ch}                                  
  \DeclareMathOperator{\size}{size}                              
  \DeclareMathOperator{\HD}{HD}                                  
  \DeclareMathOperator{\Sym}{Sym}                                
  \DeclareMathOperator{\ext}{Ext}                                
  \DeclareMathOperator{\area}{Area}                              
  
  \newcommand{\ep}{\ensuremath{\epsilon}\xspace}

  \newcommand{\param}{{\mathchoice{\mkern1mu\mbox{\raise2.2pt\hbox{$
  \centerdot$}}
  \mkern1mu}{\mkern1mu\mbox{\raise2.2pt\hbox{$\centerdot$}}\mkern1mu}{
  \mkern1.5mu\centerdot\mkern1.5mu}{\mkern1.5mu\centerdot\mkern1.5mu}}}

  \theoremstyle{plain}
    \newtheorem{theorem}{Theorem}[section]
    \newtheorem{proposition}[theorem]{Proposition}
    \newtheorem{corollary}[theorem]{Corollary}
    \newtheorem{lemma}[theorem]{Lemma}
    \newtheorem{claim}[theorem]{Claim}

  \theoremstyle{definition}
    \newtheorem{definition}[theorem]{Definition}

  \newtheorem{introthm}{Theorem}
  
  \newtheorem{introcor}[introthm]{Corollary}

  \newcommand{\secref}[1]{\S~\ref{Sec:#1}}
  
  \newcommand{\thmref}[1]{Theorem~\ref{Thm:#1}}
  \newcommand{\propref}[1]{Proposition~\ref{Prop:#1}}
  \newcommand{\lemref}[1]{Lemma~\ref{Lem:#1}}
  \newcommand{\corref}[1]{Corollary~\ref{Cor:#1}}
  \newcommand{\figref}[1]{Figure~\ref{Fig:#1}}
  \newcommand{\eqnref}[1]{Equation~\ref{Eq:#1}}
  \newcommand{\defref}[1]{Definition~\ref{Def:#1}}

\begin{document}


  \title      {The Diameter of the Thick Part of Moduli Space and Simultaneous
  Whitehead Moves}
  \author     {Kasra Rafi}
  \address    {Department of Mathematics\\
               University of Oklahoma\\
               Norman, OK 73019-0315, USA}
  \email      {rafi@math.ok.edu}
  \author     {Jing Tao}
  \address    {Department of Mathematics\\
               University of Utah\\
               Salt Lake City, UT 84112-0090, USA}
  \email      {jing@math.utah.edu}
  \date       {}
  \author    {Kasra Rafi\footnote{\small Partially supported by NSF
  Research Grant, DMS-1007811.} $\,$ and Jing Tao}

  \maketitle
  
  \begin{abstract}

  Let $S$ be a surface of genus $g$ with $p$ punctures with negative Euler
  characteristic. We study the diameter of the $\ep$--thick part of moduli
  space of $S$ equipped with the \Teich or Thurston's Lipschitz metric. We
  show that the asymptotic behaviors in both metrics are of order $\log
  \left( \frac{g+p}{\ep} \right)$. The same result also holds for the
  $\ep$--thick part of the moduli space of metric graphs of rank $n$
  equipped with the Lipschitz metric. The proof involves a sorting
  algorithm that sorts an arbitrary labeled tree with $n$ labels using
  simultaneous Whitehead moves, where the number of steps is of order
  $\log(n)$. As a related combinatorial problem, we also compute, in the
  appendix of this paper, the asymptotic diameter of the moduli space of
  pants decompositions on $S$ in the metric of elementary moves.
  
  \end{abstract}



\section{Introduction}
 
  \label{Sec:Introduction}
  
  Let \PM be the moduli space of complete finite-volume hyperbolic surfaces
  of genus $g$ with $p$ \emph{labeled} punctures. We equip \PM with the
  \Teich metric \dT. Let $\ep_M$ be the Margulis constant (two curves of
  length less than $\ep_M$ on a hyperbolic surface do not intersect) and
  let $\ep \leq \ep_M$. Let  \PMthick be the $\ep$--thick part of $\PM$,
  that is, the space of surfaces where the length of every essential closed
  curve is at least $\ep$. By a theorem of Mumford, \PMthick is compact. We
  are interested in a better understanding of the ``shape'' of \PMthick. As
  a first step, we study the asymptotic behavior of the \Teich diameter of
  \PMthick as $g$ and $p$ go to infinity. In this paper, we prove:
  
  \begin{introthm} \label{Thm:IntroDiamModuli}
     
     There exists $K$ such that, for every $g$ and $p$ with $2g-2+p>0$ and
     every $\ep \le \ep_M$,
     \begin{equation} \label{Eq:OrderLog} 
        \frac{1}{K} \log \left(\frac{g+p}{\ep}\right) \le \diamT
        \big(\PMthick\big) \le K \log \left(\frac{g+p}{\ep}\right).
     \end{equation}
  
  \end{introthm} 
  
  We will adopt a shorthand notation and rewrite \eqnref{OrderLog} as
  $\diamT \big(\PMthick\big) \asymp \log \left( \frac{g+p}{\ep} \right)$. 
  
  There are several variations of this theorem. First, there are two
  natural possible metrics on \PMthick. For any two points $X,Y \in
  \PMthick$, we can consider the \Teich distance between $X$ and $Y$ or the
  induced path metric distance in \PMthick between them. Also, the moduli
  space has an alternative definition where the punctures are not marked.
  This is equivalent to considering the quotient space $\PM/\Sym_p$ where
  the symmetric group $\Sym_p$ acts on $\PM$ by permuting the labeling of
  the punctures. Our theorem holds for both spaces and in both senses of
  diameter. 
  
  As we shall see, an essential component of the proof of our theorem is
  that the \Teich metric is an $L^\infty$ metric. Hence, another variation
  is to consider another $L^\infty$ metric on \PM: the asymmetric Lipschitz
  metric $\dL$ as defined by Thurston \cite{thurston:MSM}. There is a
  simple inequality relating the two metrics (see \secref{TwoMetrics}). For
  any $X, Y \in \PM$,
  \begin{equation*} \label{Eq:Metrics} 
     \frac{1}{2}\dL(X,Y) \le \, \dT(X,Y).
  \end{equation*} 
  For the proof of \thmref{IntroDiamModuli}, we will in fact use the
  Lipschitz metric to obtain the lower bound and the \Teich metric to
  obtain the upper bound. In view of the above equation, this will
  simultaneously establish the same asymptotics for the diameter of
  \PMthick in both metrics. Henceforth, when we say distance or
  $d(\param,\param)$, diameter or $\diam(\param)$, without reference to a
  metric, we will mean either one of the two metrics. 

  Another metric to consider would be the Weil-Petersson metric on \Teich
  space. Heuristically, Weil-Petersson metric is an $L^2$--metric: the norm
  of a vector is an $L^2$--average of the amount of deformation throughout
  the surface. For this reason, two points that are distance one apart in
  the \Teich metric have Weil-Petersson distance at most $\sqrt{\area}$.
  Hence, our theorem provides an upper bound of order $\sqrt{g+p} \,
  \log(g+p)$ for the Weil-Petersson diameter of the \Teich space. This
  implies Theorem 1.2 in \cite{cavendish:WPD} where the growth rate of the
  Weil-Petersson diameter of \Teich space is studied. In this case a
  matching lower bound seems to be more difficult and remains open. 
  
  \subsection*{Width and height}

  As a general philosophy, one can study the geometry of a surface by
  decomposing it into pants along the shortest possible curves. If the
  curves are sufficiently short, then the geometry of the surface is
  essentially determined by the combinatorics of the pants decomposition,
  which is encoded by the dual graph of the pants decomposition. When a
  surface does not admit a short pants decomposition, then the lengths of
  the curves in the shortest pants decomposition and the twisting
  information along these curves are also relevant information. For these
  reasons, the proof of \thmref{IntroDiamModuli} naturally breaks down into
  two parts. One part considers the subset $\PBers \subset \PMthick$
  consisting of surfaces that can be decomposed into pants by curves of
  length $\ep_M$. We will refer to the diameter of \PBers as the
  \emph{width} of \PMthick. The other part considers the Hausdorff distance
  between \PMthick and \PBers. We represent this quantity by $\HD
  \big(\PMthick, \PBers)$ and refer to it as the \emph{height} of \PMthick. We
  prove:

  \begin{introthm}[Width of \PMthick] \label{Thm:IntroWidth} 
     
     \[ \diam \big( \PBers \big) \asymp \log(g+p) \]
  
  \end{introthm}
     
  \begin{introthm}[Height of \PMthick] \label{Thm:IntroHeight}
     
     \[ \HD \big( \PMthick, \PBers \big) \asymp \log \left( \frac{g+p}{\ep}
     \right). \] 
  
  \end{introthm}
  
  This pair of theorems can be viewed as a refinement of
  \thmref{IntroDiamModuli}. Using the triangle inequality, the upper bounds
  for the width and height of \PMthick provide the upper bound for the
  diameter of \PMthick. The lower bound for the height is a lower bound for
  the diameter. However, since our main interest is a better understanding
  of the shape of \PMthick, we include the lower bound arguments for both
  the width and the height in this paper.
  
  \subsection*{The diameter of space of graphs}

  The argument for the width involves solving a combinatorial problem of
  independent interest. By considering the dual graph of the pants
  decomposition, we can associate to every element $X \in \PBers$ a graph
  of rank $g$ with $p$ \emph{marked} valence-1 vertices and $(2g-2+p)$
  valence-3 vertices (see \secref{DualGraph}). Let \PGraph be the space of
  all such graphs. We consider the metric of \emph{simultaneous Whitehead
  moves} $d_S$ on \PGraph: a simultaneous Whitehead move on a graph is a
  composition of an arbitrary number of commuting Whitehead moves. This is
  a suitable metric for our purposes since we are considering $L^\infty$
  metrics on moduli space. We show that the spaces \PGraph (equipped
  with the $d_S$ metric) and \PBers are quasi-isometric and hence their 
  diameters are of the same order. 
  
  \begin{introthm}[Diameter of \PGraph] \label{Thm:IntroDiamGraph}
    
    \[ \diam_S \big( \Graphp \big) \asymp \diam_S \big( \PGraph \big)
    \asymp \log(g+p). \]

  \end{introthm}
  
  The main argument for the upper bound of \thmref{IntroDiamGraph} boils
  down to an efficient sorting algorithm for labeled trees using
  simultaneous Whitehead moves (\secref{Trees}). Let \Tree be the space of
  rooted binary trees with labels $0,\ldots n$, equipped with the metric
  simultaneous Whitehead moves. There is a distance-increasing embedding of
  \PGraph into \Tree, where $n=2g-2+p$ (see \secref{GraphtoTree}). We
  introduce a sorting algorithm using simultaneous Whitehead moves and
  show that the algorithm can sort any tree in \Tree in $\log(n)$ number
  of steps. This bound also gives the desired upper bound for the diameter
  of \PGraph.

  The lower bound for \thmref{IntroDiamGraph} can be obtained by a simple
  example (see \secref{Examples}). In fact, more can be said for \Bers and
  \Graph. By the work of \cite{pinsker}, a generic point in \Graph is an
  expander graph (see \defref{Expander}). We show that the surfaces
  associated to expander graphs have the following property: By a
  \emph{dividing curve} on a surface of genus $g$ we will mean a separating
  curve which divides the surface into two pieces, both of genus of order
  $g$. 

  \begin{introthm} \label{Thm:IntroExpander}
    
     For any surface in \Bers associated to an expander graph, the length
     of any dividing curve is at least of order $g$.

  \end{introthm}

  If we fix a surface $X \in \Bers$ with a dividing curve of length of
  order 1, then the above theorem implies the distance between $X$ and any
  surface in \Bers associated to an expander graph is at least of order
  $\log(g)$. This gives a lower bound of order $\log(g)$ for the diameter of 
  \Bers. However, noting that $\Bers$ is quasi-isometric to $\Graph$, 
  this also shows that expanders are not equidistributed in \Graph, 
  nor are they coarsely dense. This answers a question of Mirzakhani
  in the negative:
  
  \begin{introcor} \label{Cor:Not-Dense}
    
  Expander graphs are not coarsely equidistributed in the space \Graph equipped 
  with the metric $d_S$. In fact, \Graph contains a ball of radius $r$, where $r$ is
  on the order of the diameter of \Graph, that contains no expander graph. 

  \end{introcor}  
  
  \thmref{IntroExpander} also follows from Buser's work in
  \cite{buser:Eigenvalues}.

  \subsection*{Related combinatorial problems}

  Similar combinatorial problems have been considered previously in the
  literature. In \cite{bose:SDF}, an algorithm was established, using
  $\log(n)$ simultaneous flips, to transform any triangulation of an
  $n$--gon into any other. This can be rephrased in terms of a
  simultaneous-type metric on the space of \emph{unlabeled planar} trees.
  However, the result in \cite{bose:SDF} and \thmref{IntroDiamGraph} do not
  imply each other. The trees in \cite{bose:SDF} do not have labels and
  they require their elementary moves to preserve a given embedding of a
  tree in the plane. 
      
  Another related problem of interest is computing the diameter of the
  space of pants decompositions on a surface \sp or $\sp / \Sym_p$ up to
  homeomorphisms, equipped with the metric of elementary moves. This is
  equivalent to computing the diameter of \PGraph or \Graphp in the metric
  of Whitehead moves: two graphs have distance one if they differ by a
  single Whitehead move. One can compute these diameters using the existing
  works of \cite{boll:NRG} and \cite{thurston:RDT}, though neither the
  details nor the statements are contained in any existing literature as
  far as the authors know. (In the case of \Graph, an alternative proof is
  presented in \cite{cavendish:WH}.) We have included in this paper an appendix
  estimating the diameters of \PGraph and \Graphp in the metric of
  Whitehead moves (see \thmref{WHLabeled} and \thmref{WHUnlabeled}). What
  is worth noting is that in this metric the diameters of \PGraph and
  \Graphp are not the same.
  
  The theorem in \cite{thurston:RDT} is very general and gives a uniform upper 
  bound for the growth rate of the number of elements in a ball of radius $r$ in 
  any space of shapes when shapes are allowed to evolve through locally supported 
  elementary moves. Whitehead moves certainly fit that description. Another 
  example of interest is the mapping class group
  equipped with the word metric coming from the Lickorish generators
  \cite{Lic}. It is an immediate consequence of their theorem that the growth 
  rate of the mapping class group with appropriate set of generators is independent 
  of the complexity of the surface. 
   
  \subsection*{Outer space} 
  
  Our results can be extended to the setting of metric graphs which is of
  interest to the study of outer automorphism of free groups. Let $R_n$ be
  a wedge of $n$ circles. The \emph{moduli space of metric graphs} \X is
  the set of non-degenerate metrics graphs of volume $n$ with homotopy type
  $R_n$ (see \secref{OuterSpace}). (This is the same as the quotient of
  Outer space by the outer automorphism group of $\F$ \cite{vogtmann:OS}.)
  A well-studied metric on \X is the Lipschitz metric $\dL$, defined in the
  same way as for surfaces. The \ep--thick part \Xthick of \X consists of
  graphs with a lower bound \ep for the length of the shortest loop. We
  prove:

  \begin{introthm}[Diameter of \Xthick] \label{Thm:IntroDiamOuter}

     \[ \diam_L \big( \Xthick \big) \asymp \log \left( \frac{n}{\ep} \right). \]

  \end{introthm}

  \subsection*{Outline of the paper}

  The organization of the paper is as follows:

  \begin{itemize}
  
  \item[\secref{Prelim}] This section contains the background material for
  the paper. In this section, we will also introduce the metric of
  simultaneous Whitehead moves on graphs and explain how to associate to
  any surface in \PBers a rooted binary tree with labels. 

  \item[\secref{Trees}] This section contains the main algorithm which provides
  the upper bound for
  the diameter of \Tree in the metric of simultaneous Whitehead moves.
  As applications, we obtain the upper bound of \thmref{IntroDiamGraph} and
  \thmref{IntroWidth}. 
  
  \item[\secref{OuterSpace}] This section applies the work on trees to
  obtain the upper bound of \thmref{IntroDiamOuter}.
  
  \item[\secref{Examples}] This section is devoted to constructing some
  interesting examples of surfaces, which includes an example for
  \thmref{IntroExpander}. These examples also provide the lower bounds for
  \thmref{IntroWidth}, \thmref{IntroHeight}, and \thmref{IntroDiamGraph}.
  We also complete the proof of \thmref{IntroDiamOuter} in this section.
  
  \item[\secref{UpperHeight}] The final section of the paper contains the
  argument for the upper bound of \thmref{IntroHeight}. We also collect our
  results together to obtain \thmref{IntroDiamModuli}.

  \item[\secref{Appendix}] In this section, we compute the asymptotic
  diameter of \PGraph and \Graphp in the metric of Whitehead moves. The
  results and proofs of this section are independent of the rest of the
  paper. 
  
  \end{itemize}
  
  \subsection*{Acknowledgments}

  We thank Curt McMullen for pointing out to us the example of the double
  hairy torus. We also thank Saul Schleimer for a helpful conversation
  which contributed to a simplification of the arguments contained in the
  last section. Finally, we thank the referee for helpful comments. 

\section{Preliminaries}

  \label{Sec:Prelim}
  
  \subsection*{Moduli spaces}

  Let \sp be a connected, oriented surface of genus $g$ with $p$ labeled
  punctures. We require the Euler characteristic $\chi(\sp)=2-2g-p$ to be
  negative. Let \PM be the moduli space of complete, finite-volume,
  hyperbolic surfaces of homeomorphism type \sp, up to label-preserving
  isometries. The quotient \Mp of \PM by the symmetric group $\Sym_p$ on
  $p$ letters is the moduli space of unlabeled punctured surfaces. One can
  also think of (\PM) \Mp as the quotient of the \Teich space of \sp by the
  (pure) mapping class group of \sp. We refer to \cite{gardiner:QT},
  \cite{hubbard:TT}, and \cite{farb:MCG} for more details. 

  By a \emph{curve} on \sp, we will always mean a free homotopy class of a
  simple closed curve which is not homotopic to a point or to a puncture.
  For a hyperbolic surface $X$, any curve has a unique geodesic
  representative which is the shortest in its homotopy class. Given a curve
  $\alpha$ on $X$, let $\ell_X(\alpha)$ be the length of the geodesic
  representative of $\alpha$ on $X$. A curve $\alpha$ is called a
  \emph{systole} on $X$ if $\ell_X(\alpha)$ is minimal among all curves on
  $X$. We will let $\ell(X)$ be the length of a systole on $X$. Given $\ep
  > 0$, the \emph{\ep--thick part} of \PM is \[ \PMthick = \big\{ \, X \in
  \PM \,\,\,:\,\,\, \ell(X) \ge \ep \, \big\}. \] To make \PMthick
  non-empty and connected, we consider only $\ep \le \ep_M$, where $\ep_M$
  is a fixed constant such that for any $X \in \PM$, if two distinct curves
  on $X$ have lengths less than $\ep_M$ then they are disjoint. The
  constant $\ep_M$ is called the \emph{Margulis constant} and is
  independent of $g$ and $p$. 

  \subsection*{Two $L^\infty$ metrics}

  \label{Sec:TwoMetrics}

  We consider two $L^\infty$ metrics on \PM. Let $X, Y \in \PM$, 

  \begin{itemize}

     \item \emph{\Teich metric}: \[ \dT(X,Y) = \frac{1}{2} \inf_f
     \big\{\, \log K(f) \,\,\, : \,\,\, f: X \to Y \text{ is }
     K(f)\text{--quasi-conformal } \,\big\}. \] 

     \item \emph{Lipschitz metric}: \[ \dL(X,Y) = \inf_f \big\{\, \log L(f)
     \,\,\, : \,\,\, f: X \to Y \text{ is } L(f)\text{--Lipschitz }
     \,\big\}. \]
    
  \end{itemize} 
  The Lipschitz metric was introduced by Thurston in \cite{thurston:MSM}.
  Unlike the \Teich metric, the Lipschitz metric is not symmetric and one
  needs to be careful when choosing the order of the two points when computing
  the distance. Both metrics induce the same topology on \PM. We have the
  following inequality:
  \begin{equation} \label{Eq:TwoMetrics} 
     \frac{1}{2} d_L(X,Y) \le \dT(X,Y). 
  \end{equation} 
  \eqnref{TwoMetrics} follows from two facts. The first fact, due to
  Wolpert, asserts that under any $K$--quasiconformal map, the hyperbolic
  length of any curve changes by at most a factor of $K$ \cite[Lemma
  3.1]{wolpert:LS}. The second fact is due to Thurston \cite{thurston:MSM}:
  
  \begin{theorem}[Thurston] 
  
  For any $X, Y \in \PM$, 
     \[ 
       \dL(X,Y) = \sup_\alpha \inf_f
       \, \log \frac{\ell_Y \big(f(\alpha) \big)}{\ell_X(\alpha)},
     \] 
  where the $\sup$ is taken over all curves on $X$ and the $\inf$ is taken
  over all Lipschitz maps from $X$ to  $Y$.
  \end{theorem}
  
  To compute distances in the \Teich metric, Kerckhoff has a similar
  formula using extremal lengths of curves \cite{kerckhoff:AG}. For any $X
  \in \PM$, the \emph{extremal length} of a curve $\alpha$ on $X$ is
  defined to be 
  \[ 
     \ext_X(\alpha) := \sup_{\rho} \frac{\ell_\rho (\alpha)^2}{\area(\rho)}.
  \] 
  Here, $\rho$ is any metric in the conformal class of $X$,
  $\ell_\rho(\alpha)$ is the $\rho$--length of the shortest curve in the
  homotopy class of $\alpha$, and $\area(\rho)$ is the area of the surface
  $X$ equipped with the metric $\rho$. 
  
  \begin{theorem}[Kerckhoff] \label{Thm:Ker}
     
     For any $X, Y \in \M$,
     \[ \dT (X,Y) = \frac{1}{2} \sup_{\alpha} \inf_{f}\, \log
     \frac{\ext_Y \big( f(\alpha) \big)}{\ext_X(\alpha),}
     \]
     where the $\sup$ is taken over all curves on $X$ and the $\inf$ is
     taken over quasi-conformal maps from $X$ to $Y$.

  \end{theorem}
  
  \subsection*{Notations}
  
  Throughout this paper, we will use the following set of notations. Given
  two quantities $a$ and $b$, we will write $a = O(b)$ to mean $a \le K b$,
  for some uniform constant $K$. Similarly, $a = \Omega(b)$ if $a \ge K b$.
  We will say $a$ is \emph{on the order} of $b$ and write $a \asymp b$ if
  $a = O(b)$ and $a = \Omega(b)$. To control notations in a string of
  inequalities, it is sometimes convenient to replace $a = O(b)$ by $a
  \prec b$, and $a = \Omega(b)$ by $a\succ b$.

  \subsection*{Pants decomposition and Bers' constant}
  
  Two isotopy classes of curves on \sp will be called disjoint if they have
  disjoint representatives. A \emph{multicurve} on \sp is a (non-empty)
  collection of distinct curves on \sp which are pairwise disjoint. A
  \emph{pants decomposition} $P$ of \sp is a multicurve such that each
  component of $\sp \setminus P$ is a 3-holed sphere, also called a \emph{pair of
  pants}. The number of curves in $P$ is equal to the \emph{complexity}
  $\xi(\sp) = 3g-3+p$ and the number of pants in a decomposition is equal
  to $\big| \chi(\sp) \big| = 2g-2+p$. 

  Let $B(X)$ be the minimal number such that $X$ admits a pants
  decomposition $P$ with $\ell_X(\alpha) \le B(X)$ for all $\alpha \in P$.
  Let \[ B_{g,p} = \sup_{X \in \PM} B(X) \] be the \emph{Bers' constant}
  for \sp. It is originally proved by Bers that $B_{g,p}$ is finite for all $g$ and
  $p$. For closed surfaces, Buser gave an explicit upper and lower bounds
  for $B_g$ in \cite[\S 5.2]{buser:GSC}:

  \begin{theorem}[Buser] \label{Thm:Buser}

     \[ \sqrt{6g} - 2 \le B_g \le 21(g-1) \]

  \end{theorem}
  
  One can extend the proof of \thmref{Buser} to obtain upper and lower
  bound for the Bers' constant in the punctured case as well: $\Omega \big(
  \sqrt{g+p} \big) =B_{g,p} = O(g+p)$. The lower bound for $B_{g,p}$
  is obtained by Buser's hairy 
  torus construction. His construction does not lie in the thick part of moduli space. We
  give another construction in the thick part that gives the same lower bound (see
  \lemref{HairyTorus}). (In \cite{parlier:BPS},
  Parlier-Balacheff improved the upper bound of $B_{0,p}$ to match the lower bound, 
  but we will not need that here.)

  \subsection*{Width and height of \PMthick}

  \label{Sec:WidthHeight}

  Let $\PBers \subset \PMthick$ be the set of surfaces $X \in \PBers$ such
  that $X$ admits a pants decomposition where the length of every curve is
  exactly $\ep_M$. By our choice of $\ep_M$, intersecting curves cannot 
  have lengths $\ep_M$, therefore such a pants decomposition 
  on $X$ is unique. 
  
  Let $\diam_L \big( \PBers \big)$ be the maximal Lipschitz distance
  between any two points in \PBers. We will call this quantity the
  Lipschitz \emph{width} of \PMthick.  Let $\HD_L \big( \PMthick, \PBers
  \big)$ be the Lipschitz Hausdorff distance between \PMthick and \PBers,
  defined to be 
  \[
     \HD_L \, \big( \PMthick,\PBers \big) = \sup_{Y\in \PMthick} \,
     \inf_{X \in \PBers} \, \max \big\{ \dL(X,Y), \dL(Y,X) \big\}. 
  \] 
  This quantity will be called the Lipschitz \emph{height} of \PMthick. The
  \Teich width and height of \PMthick are defined similarly. 

  \subsection*{Dual graphs to pants decompositions}

  \label{Sec:DualGraph}

  To compare geometries of various surfaces in \PBers, we can look at the
  dual graph of the pants decompositions. Given $X \in \PBers$, let $P$ be
  the associated pants decomposition on $X$. The \emph{dual graph}
  $\Gamma_P$ of $P$ has a vertex for each pair of pants in $X \setminus P$
  or for each puncture of $X$. Two (not necessarily distinct) vertices are
  connected by an edge if either they represent two (not necessarily
  distinct) pairs of pants glued along some curve in $P$, or if one
  vertex is a puncture contained in the pants represented by the other
  vertex. See \figref{DualGraph} for an example in genus $8$.
  
  \begin{figure}[!htp]   
  \begin{center}
  \includegraphics{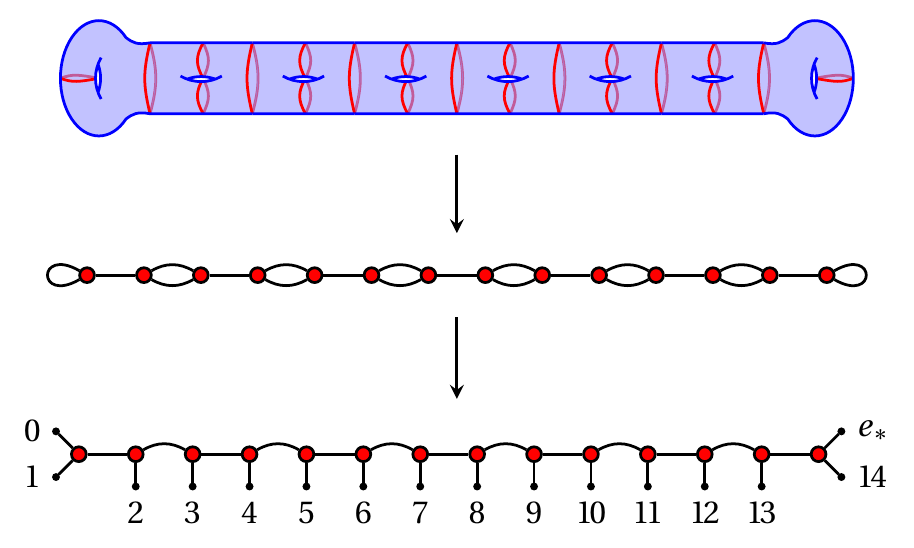}
  \end{center} 
  \caption{From pants decomposition to trivalent graph; 
  from trivalent graph to labeled tree.} 
  \label{Fig:DualGraph} 
  \end{figure}
  
  For closed surfaces of genus $g$, the dual graph to a pants decomposition
  is always a trivalent graph of rank $g$. For punctured surfaces, the dual
  graph is a graph of rank $g$ with $p$ \emph{marked} valence--1 vertices
  and $2g-2+p$ valence--3 vertices. Let \PGraph be the set of all such
  graphs. Let 
  \[ 
     \psi : \PBers  \to \PGraph,
  \]
  be the map defined by the dual graph construction. The map $\psi$ is
  surjective. Shearing along each pants curve does not change the dual
  graph, thus each fiber is a ($3g-3+p$)--dimensional torus, and since each
  pants curve has length $\ep_M$, each fiber has uniformly bounded
  diameter. 

  \subsection*{Whitehead moves on graphs}
  
  \label{Sec:Whitehead}

  Given $\Gamma \in \PGraph$, we will call an edge $e$ of $\Gamma$
  \emph{interior} if both vertices of $e$ have valence three; otherwise $e$
  is \emph{exterior}. From any $\Gamma$, there is a way of deriving a new
  graph by modifying the local gluing structure about an interior edge $e$,
  called a \emph{Whitehead move} on $e$. A Whitehead move on
  $e$ is a process of collapsing $e$ and reopening in a different
  direction. We allow two ways to reopen, as illustrated in
  \figref{WhiteheadMove} on the left. 
  
  \begin{figure}[!htp]   
  \begin{center}
  \includegraphics{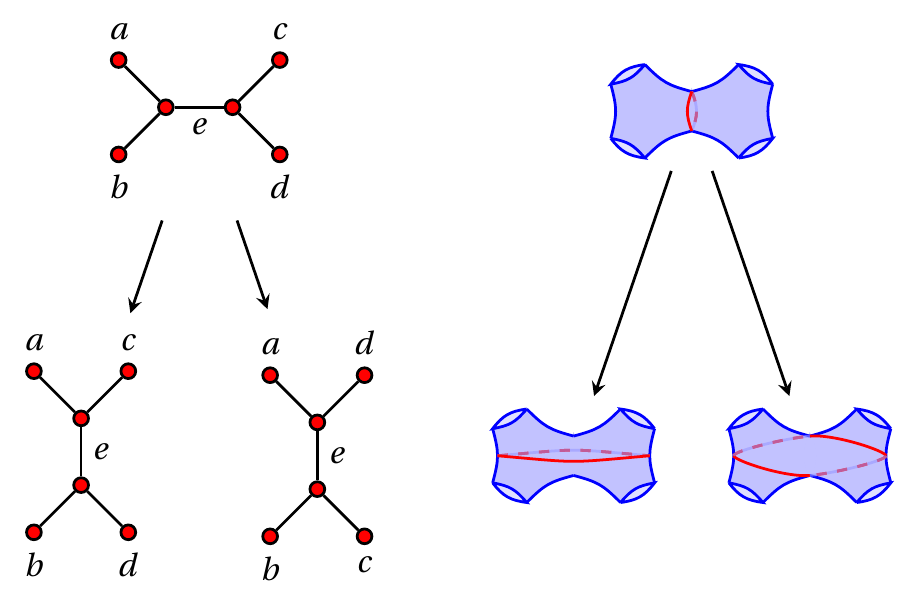}
  \end{center} 
  \caption{Whitehead on $e$ on left;
  corresponding elementary move on right.} 
  \label{Fig:WhiteheadMove}
  \end{figure}

  One may endow \PGraph with the metric $d_W$ of Whitehead moves:
  $d_W(\Gamma_1, \Gamma_2) = 1$ if and only if $\Gamma_1$ and $\Gamma_2$
  differ by a Whitehead move. Since our graph $\Gamma$ corresponds to a
  pants decomposition $P$ on a surface, there is a natural interpretation
  of Whitehead moves as \emph{elementary moves} on pants decompositions.
  Each edge $e$ in $\Gamma$ corresponds to two pairs of pants glued along a
  common curve $\alpha$. An elementary move on $\alpha$ changes $P$ by
  fixing all curves in $P \setminus \alpha$, and replacing $\alpha$ by a
  transverse curve that intersects it minimally.  Since only Whitehead
  moves associated to interior edges with distinct vertices are
  non-trivial, we only consider the case where $\alpha$ lies in a $4$-holed
  sphere. The two different directions of reopening represent the two
  minimally-intersecting transverse curves up to homeomorphisms of the
  surface. (See the right hand side of \figref{WhiteheadMove}.) Therefore,
  \PGraph equipped with the metric of Whitehead moves is isometric to the
  space of homeomorphism types of pants decompositions on \sp equipped with
  the metric of elementary moves. 
  
  \subsection*{Simultaneous Whitehead moves}
  
  \label{Sec:Simultaneous}

  Since we are
  considering $L^\infty$ metrics on moduli spaces, the metric of Whitehead
  moves on \PGraph is not the correct model metric for \PBers. 
  Our goal is to equip \PGraph with an appropriate metric so that its
  diameter is of the same order as the diameter of \PBers.  Making one
  Whitehead move corresponds to modifying a surface $X \in \PBers$ by an
  elementary move in a four-holed sphere. But modifying $X$ in several
  disjoint four-holed spheres at the same time contributes the same amount
  of distortion after taking the sup. This observation leads to the
  definition of \emph{simultaneous Whitehead moves}.

  For any graph $\Gamma \in \PGraph$, we will say two edges of $\Gamma$ are
  \emph{disjoint} if they do not share any vertices. Note that Whitehead
  moves on disjoint edges commute. Disjoint edges in $\Gamma$ correspond to
  four-holed spheres which have disjoint interiors. Hence the corresponding
  elementary moves also commute with each other. A \emph{simultaneous
  Whitehead move} on $\Gamma$ is a composition of Whitehead moves on an
  arbitrary number of pairwise disjoint edges in $\Gamma$. 

  We equip \PGraph with the metric $d_S$ of simultaneous Whitehead moves:
  $d_S (\Gamma_1,\Gamma_2) = 1$ if and only if $\Gamma_1$ and $\Gamma_2$
  differ by a simultaneous Whitehead move. The following
  lemma allows us to bound distances in \PBers by distances in
  \PGraph. Recall the dual graph map $\psi : \PBers \to \PGraph$. 

  \begin{lemma} \label{Lem:WidthAboveGraph}

     There exists a uniform constant $K$ such that for any $g$ and $p$, if
     $X, Y \in \PBers$, then 
     \[ 
        \dT(X,Y) \le \, K d_S \big(\psi(X), \psi(Y) \big). 
     \] 
  \end{lemma}

  \begin{proof}
    
     Let $S$ be a hyperbolic surface of genus $0$ with $4$ geodesic
     boundary components $\gamma_1, \ldots, \gamma_4$ such that each
     $\gamma_i$ has length $\ep_M$ and, furthermore, $S$ contains a
     non-peripheral curve $\alpha$ of length $\ep_M$. The moduli space of
     all such surfaces is compact, since there are only finitely many ways
     $\alpha$ can separate the curves $\gamma_i$ into two groups, and the
     amount of shearing along $\alpha$ is bounded by its length. Thus,
     there exists a uniform constant $K_0$, such that for any other such
     surface $S'$ with boundaries $\gamma_i'$ and an essential curve
     $\alpha'$ of length $\ep_M$, there is a $K_0$--quasi-conformal map $S
     \to S'$ taking each $\gamma_i$ to $\gamma_i'$ and $\alpha$ to
     $\alpha'$. 
     
     Similarly, we can consider the space of hyperbolic surfaces
     homeomorphic to a $4$--holed sphere, with $a$ punctures and $b$
     geodesic boundaries of length $\ep_M$ so that $a+b=4$. Let $K_a$ be a
     uniform constant depending only on $a$ such that there is
     $K_a$--quasi-conformal map between any two such surfaces fixing the
     boundaries and the punctures. Let $K'$ be the maximum of the $K_a$'s. 
        
     Now suppose $X, Y \in \PBers$ have $ d_S \big( \psi(X), \psi(Y) \big)
     = 1$. This means that there is a set of disjoint four-holed spheres in
     $X$ on which we need to make a modification as above. We can construct
     a $K'$--quasi-conformal map from $X$ to some surface $Y'$, locally
     using maps as above, where $Y$ and $Y'$ have the same dual graph. That
     is, $\dT(X,Y) \le \frac{1}{2} \log K'$ and $\psi(Y)= \psi(Y')$. But,
     as mentioned before, the preimage of a point under $\psi$ is a compact
     set with uniform diameter. Hence, $\dT(Y,Y') = O(1)$. By the triangle
     inequality, we have $\dT(X,Y) \le K$, for some uniform $K$. 
     
     In the general situation where $d_S \big( \psi(X), \psi(Y) \big) = n$.
     Let $X = X_0, X_1, \ldots, X_n = Y$ be a sequence of elements in
     \PBers with $d_S \big( \psi(X_{i-1}), \psi(X_i) = 1$ for all
     $i=1,\ldots n$. From above, $d_T (X_{i-1}, X_i) \le K$. Thus, by the
     triangle inequality, we obtain the desired statement of the lemma: 
     \[ d_T(X,Y) \le \sum_i^n d_T(X_{i-1},X_i) \le Kn = K d_S \big(
     \psi(X), \psi(Y) \big). \qedhere\]
  \end{proof}
  
  We conclude:

  \begin{corollary} \label{Cor:WidthAboveGraph}
    
     \[ 
        \diam_T \big(\Bersp\big) \le \diam_T \big(\PBers\big) \prec 
        \diam_S \big( \PGraph \big). 
     \]

  \end{corollary}
  
  \subsection*{From graphs to labeled trees}
  
  \label{Sec:GraphtoTree}
  
  By \corref{WidthAboveGraph}, the problem of bounding the width of
  \PMthick from above can be replaced by the problem of bounding $\diam_S
  \big( \PGraph \big)$ from above. For the latter problem, it will be more
  convenient to cut each graph into a binary tree and label the ends in
  pairs to remember the gluing. We explain how to do this after some
  definitions.

  By a \emph{binary tree} (or a tree for short) we mean a connected graph
  with no loops, so that the valence at each vertex is either one or three.
  A \emph{rooted} tree has a distinguished exterior edge $e_*$. A
  \emph{labeled} tree is a rooted tree where all ends (exterior edges
  except for $e_*$) are labeled with numbers $\{0, 1, 2, \ldots, n \}$,
  where $n$ will be called the \emph{complexity} of the tree. Note that a
  tree of complexity $n$ has $n+2$ exterior edges and $n-1$ interior edges.
  Two labeled trees are said to be equal if there is a homeomorphism
  between them taking root to root preserving the labels. After this
  identification, there are finitely many labeled trees for each fixed
  complexity $n$. Let \Tree be the space of labeled trees of complexity
  $n$. Whitehead moves or simultaneous Whitehead moves for trees are
  defined as before. We retain the notations $d_W$ and $d_S$ for the
  metrics of Whitehead moves and simultaneous Whitehead moves on \Tree,
  respectively. 
  
  We now construct a map $\PGraph \to \Tree$, $n=2g-2+p$. Let $\Gamma \in
  \PGraph$. The graph $\Gamma$ has $p$ exterior edges which are labeled
  (inherit the labeling from the marked vertices). We may identify the
  labeling as $0, \ldots, p-1$. Now arbitrarily pick a spanning tree in
  $\Gamma$. The complement of the spanning tree in $\Gamma$ contains
  exactly $g$ edges. We cut each such edge in half, resulting in a tree $T$
  with $2g$ unlabeled exterior edges. We label these edges arbitrarily from
  $p, \ldots 2g+p-1$ with the only restriction that $p+2k$ is glued to
  $p+2k+1$ in $\Gamma$, for $k=0,\ldots g-1$. Finally, erase the edge with
  the highest label, $2g+p-1$, and make that the root of $T$. The resulting
  tree $T$ is an element in \Tree associated to $\Gamma$. See
  \figref{DualGraph} for an example. 
  
  If the two trees associated to two graphs differ by a simultaneous
  Whitehead move, then the two graphs also differ by one simultaneous
  Whitehead move, hence 
  \begin{equation} \label{Eq:GraphAboveTree}
     \diam_S \big( \Graphp \big) \le \diam_S \big( \PGraph \big) \le
     \diam_S \big( \Tree \big). 
  \end{equation} 
  By \corref{WidthAboveGraph}, we have
  \begin{equation} \label{Eq:WidthAboveTree} 
     \diam_T \big(\Bersp\big) \le \diam_T \big(\PBers\big) \prec \diam_S
     \big( \Tree \big).
  \end{equation}
 
\section{Trees and upper bound on width} 

  \label{Sec:Trees}
  
  In this section, we describe two algorithms for transforming a binary
  tree into a desired shape efficiently using simultaneous Whitehead
  moves. We prove 

  \begin{theorem} \label{Thm:UpperTree}
     
     For any $n$,
     \[ \diam_S \big( \Tree \big) = O \big( \log(n) \big). \]

  \end{theorem}
  In view of \eqnref{GraphAboveTree} and \eqnref{WidthAboveTree}, we obtain
  the upper bound of \thmref{IntroDiamGraph} and the upper bound for the
  width of \PMthick. We remark that the lower bound for
  \thmref{IntroDiamGraph} is easy, but we will postpone a proof to
  \secref{Examples} where we discuss the lower bound for the Lipschitz
  width of $\Mpthick$. 

  This section is organized as follows. In \secref{ReducingHeight}, we will
  introduce an algorithm which makes any tree more compact, by reducing its
  height to be of order $\log(n)$ in $O\big(\log(n)\big)$ simultaneous
  Whitehead moves. Then in \secref{FullySortedTree}, we will introduce a
  distinguished element $T_n \in \Tree$, called the \emph{fully-sorted
  tree}. In \secref{Sorting}, we will describe how to sort the labels of a
  compact tree to be fully-sorted in $O\big( \log (n) \big)$ number of
  simultaneous Whitehead moves. 
  
  \subsection{Reducing the height} \label{Sec:ReducingHeight}
    
  Given $T \in \Tree$, the root $e_*$ defines a partial order ``$<$'' on
  the set of edges of $T$ where $e_*$ is the minimal element: Given edges
  $e_1$ and $e_2$, we say $e_2$ is a \emph{descendant} of $e_1$, and write
  $e_1<e_2$, if the path from $e_2$ to $e_*$ contains $e_1$. (We note that,
  in our figures, ``descendants'' are drawn as ``ascendants''.) If $e_1$
  and $e_2$ are adjacent and $e_1 < e_2$, then $e_1$ and $e_2$ are in
  \emph{parent-child} relationship. The maximal elements of this relation
  are called the \emph{ends} of $T$. Given any edge $e$ in $T$, let $T_e$
  be the subtree of $T$ consisting of $e$ and its descendants. We will say
  $T_e$ is \emph{rooted} at $e$ or $e$ is the root of $T_e$. The
  \emph{size} of $T_e$ will be the number of edges of $T_e$. The edge $e$
  defines a partial ordering on the edges of $T_e$ which is the one
  inherited from $T$. The maximal elements of $T_e$ will be called the ends
  of $T_e$, and the labels of $T_e$ will be the labels of the ends of
  $T_e$. We will say the root $e$ of $T_e$ has \emph{height} 1, its
  children have height 2, and inductively define the height of all edges of
  $T_e$. The maximal possible height will be called the \emph{height} of
  $T_e$.  The height of $T$ will be the height of $T_{e_*} = T$. Given a
  subtree $T_e$ of height $h$, we will say $T_e$ is \emph{full} if $T_e$
  has $2^h$ ends.  
  
  For each interior edge $e$, we will label its children $e_l$ (the left)
  and $e_r$ (the right) and we call the trees $T_{e_r}$ and $T_{e_l}$ the
  children subtrees of $T_e$. Similarly, we label the left and right edges
  of $e_l$ and $e_r$ by $e_{ll}$, $e_{lr}$, $e_{rl}$ and $e_{rr}$ and refer
  to the associated subtrees as the grandchildren subtrees. 
  
  \begin{proposition} \label{Prop:HeightTree}
    
     Any tree $T$ can be transformed to have a height $6 \log_2 (n)$ 
     after $O \big( \log(n) \big)$ simultaneous moves. 

  \end{proposition}

  \begin{proof}

     To prove the proposition, it suffices to show that, if the height of
     $T$ is larger than $6 \log_2 (n)$, one can apply one simultaneous
     move to reduce the height by a definite multiplicative factor.
          
     Let $e$ be any edge. We define a special Whitehead move, called the
     \emph{balance move at $e$}. Compare the sizes of the subtrees
     $T_{e_{ll}}$, $T_{e_{lr}}$, $T_{e_{rl}}$, and $T_{e_{rr}}$. If there
     is an absolute maximum among them, (say $T_{e_{ll}}$), we apply one
     Whitehead move to the edge $e_l$ which reduces the height of $e_{ll}$,
     does not change the height of $e_{lr}$ and increases the heights of
     both $e_{rl}$ and $e_{rr}$ (see \figref{BalanceMove}). If there is no
     absolute maximum, we apply no move. 
      
     \begin{figure}[!htp]  
     \begin{center}
     \includegraphics{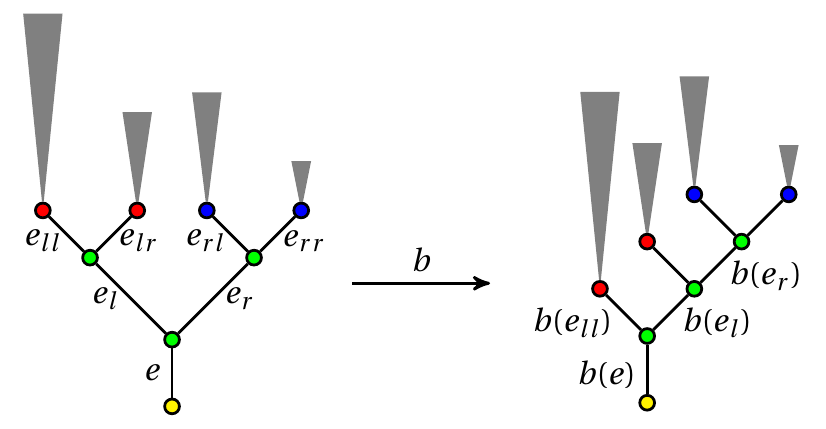}
     \end{center}
     \caption{The map $\stackrel{b}{\longrightarrow}$ is a balance move at
     $e$.} 
     \label{Fig:BalanceMove}
     \end{figure}
  
     Now consider the following simultaneous Whitehead move consisting of a
     balance move for each edge of odd height. Since the support of these
     balance moves are disjoint, their composition defines a simultaneous
     Whitehead move. We show that after one such simultaneous Whitehead
     move, any edge whose height was larger than $6 \log_2(n)$ will have
     its height reduced by at least a multiplicative factor $\frac 56$.
     That is, if the height of $T$ is larger than $6 \log_2(n)$ then it is
     reduced by a factor $\frac 56$. The algorithm stops when the height of
     $T$ is less than  $6 \log_2(n)$
     
     Let $e$ be any edge of height $h_e \ge 6 \log (n)$. Consider the path
     $P$ of length $h_e$  in $T$ connecting $e$ to the root $e_*$. The
     height of $e$ is affected by balance moves associated to edges along
     this path that have odd height. Some will decrease the height of $e$,
     some will decrease it and some will leave it unchanged. Moving down
     from $e$ to $e_*$, let $e_1, e_2, \ldots e_s$ be the set of edges in
     $P$ whose heights are odd and so that, for $j=1, \ldots, s$, the
     balance move associated to $e_j$ does not decrease the height of $e$.
     We bound the number of such edges, that is, we show that most edges
     will decrease the height of $e$. 
      
     By assumption, $T_e \subset T_{e_1}$ has at least $1$ edge. For the
     edge $e_1$ the grandchild tree containing $e$ is not of maximal size.
     Hence, $T_{e_1}$ contains at least two subtrees of size $T_{e}$ plus
     the edge $e_1$ and its children. That is, \[ \size(T_{e_1}) \geq 2
     \size(T_{e}) + 3 \ge 5. \] Similarly, at each $e_j$, the grandchild
     tree containing $e$ is not of maximal size. Hence, $T_{e_j}$ contains
     two grandchildren subtrees of size at least $T_{e_{j-1}}$. That is, \[
     \size (T_{e_j}) \geq 2 \, \size (T_{e_{j-1}}) + 3. \] By induction,
     $\size(T_{e_j}) \geq 2^{j+1}+1$. The total number of edges in $T$,
     $2n+1$, is larger than $\size(T_{e_s})$. Therefore, $$ s \le \log_2
     (n). $$
   
     After applying our simultaneous Whitehead move associate to $e_j$, $j=
     1, \ldots, s$ may cause a height increase for $e$, but a balance move
     at every other edge along $P$ with odd height results in a height
     reduction for $e$. The number of these edges that cause a height
     decrease is at least $h_e/2 -\log_2 (n)$ and the number of edges that
     may increase the height is at most $\log_2 (n)$. Hence, after the
     simultaneous move, the height of $e$ is no more than 
     \[ 
       h_e - \frac{h_e}{2} + 2\log_2 (n) \le \frac{5}{6} h_e. 
     \] 
     
     Since the maximum height of any tree is $n+1$, the number of
     simultaneous moves required to reduce the height of $T$ to a height
     less than $6 \log_2(n)$ is at most $\log_{6/5}(n+1)$. This concludes
     the proof. 
  \end{proof}
  
  \subsection{Fully-sorted tree} \label{Sec:FullySortedTree}
    
  In this section, we inductively construct a tree $T_n \in \Tree$, for
  each $n$. Considering $T_n$ as a base point of $\Tree$, an upper bound
  for the distance between any tree and $T_n$ provides an upper bound for
  $\diam_S\big(\Tree\big)$. 
  
  For $n=0$, $T_0$ is just the root edge $e_*$ with one end labeled $0$.
  Now assume that we have already constructed a tree $T_k \in
  \text{Tree}(k)$ for all $k<n$. Let $m$ be the largest number so that $2^m
  \leq n$ and let $k = n - 2^m \geq 0$. Take the root edge $e_*$. On the
  left, we attach $T_{2^m-1}$ (a full tree of height $m$) and on the right
  we attach a copy of $T_k$. We then change the labels of ends of $T_k$ by
  adding $2^m$ to their values.  The tree $T_n$ is the tree giving the
  binary expansion of numbers $0$ to $n$. See \figref{FullySortedTree} for
  some examples of $T_n$. 

  \begin{figure}[!htp] 
  \begin{center}
  \includegraphics{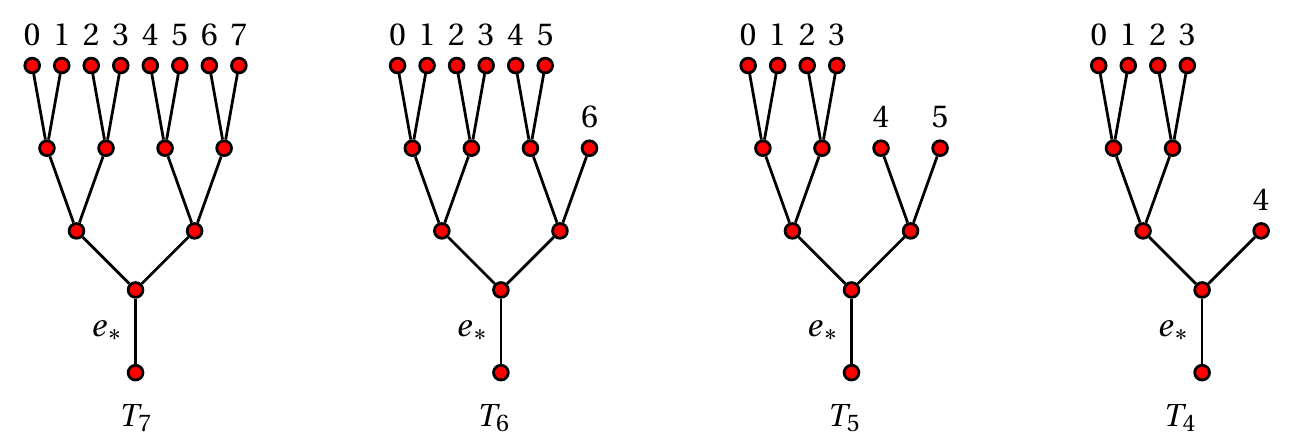}
  \end{center}
  \caption{Examples of fully-sorted trees} 
  \label{Fig:FullySortedTree} 
  \end{figure}

  We can also think of $T_n$ as a \emph{fully-sorted} tree. This is the 
  description of $T_n$ used in \propref{SortTree}. 

  \begin{definition}  \label{Def:Sorted}
   
    Let $d$ be the number of digits needed for the binary expansion of $n$.
    Take $T \in \Tree$. Write every label of $T$ as a $d$--digit number,
    possibly starting with several zeros. For an edge $e \in T$, let $n_e$
    be the number of ends of $T_e$ and $d_e$ be the number of digits in the
    binary expansion of $n_e$. By the $k$--th digit of a label, we always
    mean the $k$--th digit from the right. We  say an edge $e \in T$ is
    \emph{$k$--sorted} if the digits $(k+1)$ to $d$ of all labels of $T_e$
    are the same and either 

    \begin{enumerate} 
    
      \item[(a1)] all labels of $T_e$ have the same $k$--th digit as well
      or
       
      \item[(a2)] all labels of $T_e$ whose $k$--th digit is 0 appear as
      ends of $T_{e_l}$ and all labels of $T_e$ whose $k$--th digit is 1
      appear as ends of  $T_{e_r}$.
            
    \end{enumerate}   
    
  \end{definition}

  We say $T$ is \emph{fully-sorted} if every edge $e$ is $d_e$--sorted.
  Note that if an edge $e$ is $k$--sorted, then it is also $j$--sorted for
  all $j \ge k$. All descendants of $e$ are also at least $k$--sorted. On
  the other hand, if $k < d_e$, then $e$ cannot be $k$--sorted (there would not
  be enough free digits to represent $n_e$ different numbers). 
    
  Here is the second characterization of the base tree $T_n$.  
  
  \begin{lemma} \label{Lem:Tn}
  
    A tree $T \in \Tree$ is fully-sorted if and only if $T=T_n$. 
  
  \end{lemma}
  
  \begin{proof}
    
    The statement is clear for $n=0$. Assume $T$ is a fully-sorted tree and
    $n\geq1$. The edge $e_*$ is $d$--sorted, since $d_{e*}=d$. In this case,
    condition $(a2)$ must hold, because the labels of $T$ cannot all start with
    the same number. That is, all the labels at the ends of $T_{(e_*)_l}$
    start with the digit 0 and all the labels at the ends of $T_{(e_*)_r}$
    start with the digit 1.
     
    We now cut $e_*$ out obtaining the two children trees. We modify the labels 
    of the left tree by removing the first digit $0$ from all labels. The labels
    on the right may have several digits in common. We modify the labels by
    removing all these digits. That is, the number of digits are the
    minimum needed to represent the labels. Denote the these modified tree
    simply by $T_l$ and $T_r$. 
    
    We now check that these two trees are still fully-sorted. Consider the
    tree $T_r$ and assume that $s$ digits have been removed in the
    modifications of the labels. For $e \in T_r$, we need to show $e$ is
    $d_e$--sorted in $T_r$. Since the number of ends $n_e$ of $T_e$ does
    not change by cutting out the root $e_*$ from $T$, the edge $e$ was
    $d_e$--sorted in $T$. This means that all digits from $d_e+1$ to $d$ of
    the labels of $T_e$ are the same in $T$ and either condition 1 or 2
    held true for $e$ in $T$. Removing $s$ unnecessary digits from $T_r$
    means, viewing $T_e$ as a subtree of $T_r$, the labels must agree on
    all digits from $d_e+1$ to $d-s$. Furthermore, if conditions 1 or 2
    held true for $e$ in $T$ they would still hold for $e$ in $T_r$.
    Therefore, every edge $e$ of $T_r$ is $d_e$--sorted. The proof for
    $T_l$ is the same $(s=1)$.  
    
     Note that $T_l$ has $2^m$ ends where $m$ is the largest number with
     $n \geq 2^m$. Since $T_l$ is fully-sorted, by induction $T_l= T_{2^m-1}$.
     Similarly, $T_r= T_k$, where $k= n-2^m$, because $T_r$ has $k+1$ ends. 
     That is, $T=T_n$.    
  \end{proof}

  \subsection{Sorting} 
  
  \label{Sec:Sorting}  
  
  We now present the sorting algorithm which will transform any tree $T$ of
  height less than $6 \log_2(n)$ to a fully-sorted tree (which we know has
  to equal $T_n$) in $O\big( \log (n) \big)$ simultaneous Whitehead moves. 
  
  Note that the ends of a tree are alway $1$--sorted. Essentially our
  algorithm is to sort the tree at different digits wherever possible by
  applying a \emph{simultaneous sort move} which we describe below.  

  We say an edge $e$ is \emph{$k$--pre-sorted}, if the following conditions
  hold \begin{itemize}

  \item[(b1)] the children $e_r$ and $e_l$ are $k$--sorted,
  
  \item[(b2)] the digits $(k+1)$ to $d$ of all the labels at the ends of
  $T_e$ are the same, and
  
  \item[(b3)] the edge $e$ is not $k$--sorted. 
  
  \end{itemize}   
  We say an edge is pre-sorted if it is pre-sorted for some value of $k$. 

  To make this well defined for $k=d$ we assume that all edges are always
  $(d+1)$--sorted (one can think of the digit $(d+1)$ is always being $0$).
  Given any $k$--pre-sorted edge $e$, one can apply a \emph{sort move} at
  $e$ to make $e$ $j$--sorted, for some $j\le k$. There are essentially
  three types of such moves depicted in \figref{SortMove}. The first type
  requires three Whitehead moves supported at $e_l$ and $e_r$. The second
  and third type require only one Whitehead move supported at one of $e_l$
  or $e_r$. 
    
  \begin{figure}[!htp]
  \begin{center}
  \includegraphics{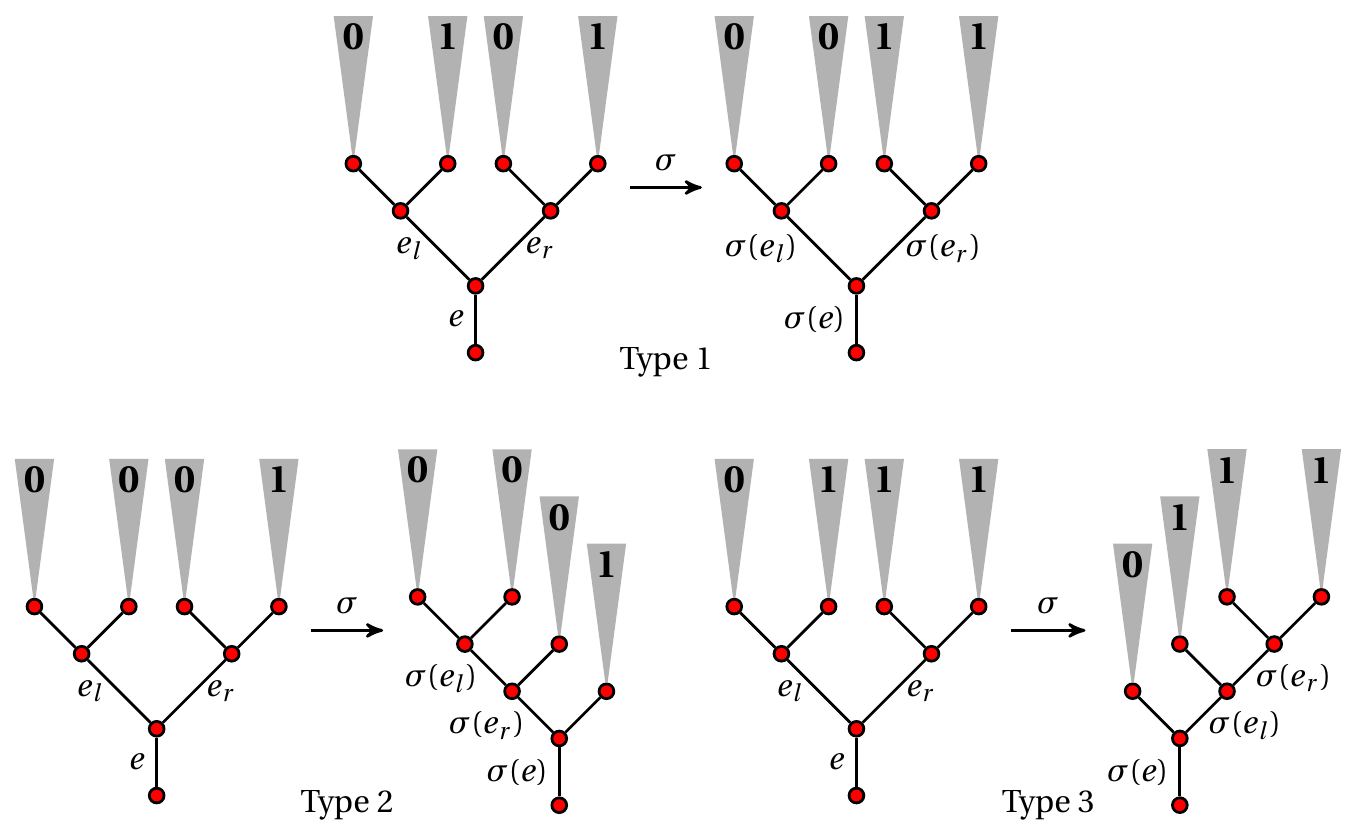}
  \end{center}
  \caption{The map $\stackrel{\sigma}{\longrightarrow}$ represents a sort
  move at $e$. There are three types of sort moves. }
  \label{Fig:SortMove}
  \end{figure}
    
  We say an edge is \emph{pre-sorted} if it is $k$--pre-sorted for some
  $k$. We claim the following statement.
  
  \begin{lemma} \label{Lem:DisjointSupport}
  
  Let $E$ be the collections of pre-sorted edges of $T$. Then the sort moves
  associated to the edges in $E$ have disjoint support.  
  \end{lemma}  
  
  \begin{proof}
  
     The support of the sorting moves at two edges is aways disjoint if the
     edges are not in a parent-child relationship. Hence, it is enough to
     show that if $e$ and $e_l$  are pre-sorted, then the support of the
     sort move at $e$ is at $e_r$. This is equivalent to showing that, in
     this case, the sort move at $e$ is of Type 2. The support of the sort
     move at $e_l$ is a subset of $\{e_{ll}, e_{lr}\}$, which are disjoint
     from $e_r$. Hence, these two moves do not interfere with one another. 

     Suppose $e$ is $k$--pre-sorted and $e_l$ is $j$--pre-sorted. Condition
     (b1) applied to $e$ implies that, $e_l$ is $k$--sorted. But $e_l$ is
     $j$--pre-sorted. Thus $j < k$. By condition $(s2)$, all digits $j+1\le
     k < d$ of the labels of $T_{e_l}$ must be the same. That is, the sort
     move at $e$ is of type $2$.
  \end{proof} 

  \begin{definition} \label{Def:Ripe}

     We will say a tree $T$ is \emph{ripe} if for any edge $e \in T$ is
     $k$--sorted, then each child of $e$ is either $(k-1)$--pre-sorted or
     $(k-1)$--sorted. Note that if $T$ is ripe, then any subtree of $T$ is
     ripe. We will say an edge $e \in T$ is \emph{ripe} if $T_e$ is ripe. 
  
  \end{definition}
  
  Now let $\overline{E} \subset E$ be the set of pre-sorted edges $e$ so
  that $T_e$ is ripe. A \emph{simultaneous sort moves} is a composition of
  sort moves associated to edges in $\bar E$. A simultaneous sort move is a
  composition of $3$ simultaneous Whitehead moves.  

  We denote by $T \stackrel{\sigma}{\longrightarrow} T'$ to mean $T'$ is
  obtained from $T$ from one simultaneous sort move. There is a natural
  identification of the edges in $T$ with the edges in $T'$ via the map
  $\sigma$. We denote the edge in $T'$ associate to an edge $e \in T$ by
  $\sigma(e)$. 
  
  \begin{lemma} \label{Lem:Ripe}
     
     Suppose $T \stackrel{\sigma}{\longrightarrow} T'$ and let $\sigma(e)$ be
     an edge in $T'$. If the children of $e$ are ripe in $T$ then $\sigma(e)$
     is ripe in $T'$. 
     
  \end{lemma}
   
  \begin{proof}

     Let $k$ be the minimal number such that $\sigma(e)$ is $k$--sorted. We
     need to show that each child of $\sigma(e)$ is either
     $(k-1)$--pre-sorted or $(k-1)$--sorted.  We argue in five cases
     depending on how the local picture around $e$ changes under $\sigma$.
     In each case, the lemma essentially follows from the definition. 
      
     \medskip \noindent {\bf Case 1:} $e$ is pre-sorted and the sorting
     move is of type 1. 
     
     In this case, the children of $\sigma(e)$ are images of the children
     of $e$ and the grandchildren of $e$ are mapped to the grandchildren of
     $\sigma(e)$. Since $\sigma(e)$ is $k$--sorted, the digits $(k+1)$ to
     $d$ of the ends of $T'_{\sigma(e)}$ match and the $k$--th digits are
     as depicted in \figref{SortMove}. This means $e_r$ and $e_l$ are
     $k$--sorted. But $e_r$ and $e_l$ are ripe, hence the grandchildren of
     $e$ are $(k-1)$--sorted or $(k-1)$--pre-sorted. Thus, the
     grandchildren of $\sigma(e)$ are $(k-1)$--sorted. Hence, the children
     of $\sigma(e)$ are either $(k-1)$--sorted or $(k-1)$--pre-sorted (the
     conditions (b1) and (b2) holds but (b3) may or may not hold). That is,
     $\sigma(e)$ is ripe in $T'$.     
            
     \medskip \noindent {\bf Case 2:} $e$ is pre-sorted and the sorting
     move is of type 2 or 3.      
     
     By symmetry, we may assume type 2. In this case, a child of
     $\sigma(e)$ is an image of either $e_r$ or $e_{rr}$. First consider
     $\sigma(e_r)=\sigma(e)_l$. As before, $e_l$ is $k$--sorted and its
     children are at least $(k-1)$--presorted. But since the $k$--th digit
     of labels at the ends of $T_{e_l}$ match, $e_{ll}$ and $e_{lr}$ are in
     fact at l east $(k-1)$--sorted. Hence, $e_l$ is either $(k-1)$--sorted
     or $(k-1)$--presorted and $\sigma(e_l)$ is at least $(k-1)$--sorted.
     (see \figref{SortMove}). 
     
     Note also that, $\sigma(e)_{lr}$ is an image of a grandchild of $e$
     and, as argued in previous case, it is at least $(k-1)$--sorted. Thus,
     the children if $\sigma(e)_l$ are both at least $(k-1)$--sorted and
     hence $\sigma(e)_l$ is either $(k-1)$-pre-sorted or $(k-1)$--sorted. 
     
     The argument is easier for $\sigma(e_{rr})=\sigma(e)_r$ since
     $\sigma(e)_r$ is an image of a grandchild of $e$ and hence it is
     $(k-1)$--sorted. Therefore, $\sigma(e)$ is ripe.
     
     \medskip \noindent {\bf Case 3:} $e$ is not pre-sorted and the
     children of $e$ are mapped to the children of $\sigma(e)$. 
     
     In this case, $e$ is as sorted as $\sigma(e)$. Hence, $e_r$ and $e_l$
     are at least $k$--sorted and, since they are ripe, the grandchildren
     are $(k-1)$--sorted or $(k-1)$--pre-sorted. That is, the children of
     $\sigma(e)$ are either $(k-1)$--pre-sorted or $(k-1)$--sorted. This
     implies that $\sigma(e)$ is ripe. 
     
     \medskip \noindent {\bf Case 4:} $e$ is not pre-sorted but $\sigma$
     contains a sorting move associate to the parent of $e$ of type 1. 
         
     Let $f$ be the parent of $e$. The sorting move swaps the grandchildren
     of $f$. Since $\sigma$ contains a sorting move associated to $f$, all
     descendants of $f$ are ripe. 
    
    Since $\sigma(e)$ is $k$--sorted, the digits $(k+1)$ to $d$ of ends of
    $T'_{\sigma(e)}$ match. That means the pre-image of the children of
    $\sigma(e)$ are $k$--sorted. As before, using ripeness, we have that
    the children of $\sigma(e)$ are $(k-1)$--sorted or $(k-1)$-pre-sorted.
    Hence $\sigma(e)$ is ripe.  
     
     \medskip \noindent {\bf Case 5:} $e$ is not pre-sorted and $\sigma$
     contains a sorting move associate to the parent of $e$ of type 2 or 3. 
     
     Again, by symmetry, we may assume type 2. Let $f$ be the parent of
     $e$. The case $e= f_l$ is already covered in case 3. Assume $e=f_r$.
     From the figure, we have that all the children of $\sigma(e)$ have the
     same $k$--th digits. The proof now follows identically to case 1.  
  \end{proof}   

  \begin{proposition} \label{Prop:SortTree}

     Let $T_n \in \Tree$ be the fully-sorted tree and $T \in \Tree$ be
     any tree of height at most $6 \log_2(n)$. Then  
     \[ 
       d_S(T, T_n) = O\big(\log(n)\big). 
     \]

  \end{proposition}

  \begin{proof}

     We will show that $T$ can be transformed to $T_n$ in $O\big( \log(n)
     \big)$ simultaneous sort moves. 

     Let $h$ be the height of $T$. First, we show that $e_*$ will be
     $d$--sorted and every edge is ripe after $(h-3)$ steps. At the
     beginning, every edge is $(d+1)$--sorted and edges at the ends are
     $d$--sorted. In fact, the edges whose children are ends are also
     $d$--sorted after relabeling left and right edges. Hence, every edge
     at height $(h-2)$ is either $d$--sorted or $d$--pre-sorted and ripe.
     After the first step, every edge at height $(h-2)$ or higher is
     $d$--sorted and every edge at height $(h-3)$ is either $d$--sorted or
     $d$--pre-sorted and, by \lemref{Ripe}, ripe. Note that, if $T_e$ is
     not ripe, there are no sorting moves associated edges in the path
     connecting $e$ to $e_*$. That is, this path is preserved identically
     under $\sigma$ and in particular, the height of $e$ does not change.
     Hence, the maximum height of an edge $e$ where $T_e$ is not ripe goes
     down by at least $1$ after every simultaneous sorting move. 
     
     Continuing in this way, we get that after $(h-3)$ steps, every edge at
     height $\big( h-2 - (h-3) \big)=1$ or higher is $d$--sorted and ripe.
     That is, $e_*$ is $d$--sorted and ripe. Let $T^1$ be the resulting
     tree. We have shown that $d_s(T,T^1) \le h-3$.
     
     We now claim that if a tree $T^1 \in \Tree$ has the property that its
     root $e_*$ is $d_{e_*}=d$--sorted and $T^1$ is ripe, then $T^1$ will
     be fully-sorted after at most $d=\big\lceil \log_2(n) \big\rceil$
     simultaneous sorting moves. We will prove this by induction on $n$.
     When $n=1$ there is nothing to prove. Now suppose $n>1$. By
     assumption, $T_1$ is ripe and $e_*$ is $d$--sorted. Therefore, any
     future sorting move will preserve the children subtrees of $e_*$. Let
     $e$ be a child of $e_*$. The subtree $T^1_e$ is also ripe, since it is
     a subtree of $T_1$. After removing all digits which are common to all
     labels of $T^1_e$, $e$ is either $d_e$--sorted or $d_e$--pre-sorted.
     After applying one sorting move $T^1_e
     \stackrel{\sigma}{\longrightarrow} T^{1'}_e$, the root $e$ becomes
     $d_e$--sorted and $T^{1'}_e$ remains ripe. By the induction
     hypothesis, $T^{1'}_e$ can be transformed to $T_{d_e}$ in $d_e \le
     d-1$ simultaneous sorting moves. Therefore, after at most $d$
     simultaneous sorting moves, both subtrees attached to $e_*$ are
     fully-sorted. This exactly means that $T^1$ is fully-sorted after at
     most $d$ sorting moves. 

     We conclude: \[ d_S(T,T_n) \le d_S(T,T^1) + d_S(T^1,T_n) \le (h-3) + d
     = O \big( \log(n) \big). \qedhere\]
  \end{proof}

  This completes the proof of \thmref{UpperTree}. Using
  \eqnref{GraphAboveTree} and \eqnref{WidthAboveTree}, we deduce the
  following respective corollaries.

  \begin{corollary} \label{Cor:UpperGraph}
    
    \[ \diam_S \big( \Graphp \big) \le \diam_S \big( \PGraph \big) = O
    \big(\log(g+p) \big).\]

  \end{corollary}
  
  \begin{corollary} \label{Cor:UpperWidth}
    
    \[ \diam_T \big( \Bersp \big) \le \diam_T \big( \Bers \big) = O
    \big(\log(g+p) \big).\]

  \end{corollary}

\section{Application to the moduli space of metric graphs} 

  \label{Sec:OuterSpace}
  
  We now give an application of the algorithms in \secref{Trees} to the
  moduli space of metric graphs. Our goal is to prove the upper bound of
  \thmref{IntroDiamOuter} of the introduction. The lower bound is worked
  out in the next section.

  \subsection{Lipschitz diameter of moduli space of metric graphs}

  \label{sec:DiamOuter}

    Let $R_n$ be a wedge of $n$ circles. Let \X be the space of isometry
    classes of metric graphs $G$ with the following properties:

    \begin{itemize}

       \item $G$ is homotopy equivalent to $R_n$.

       \item The valence of each vertex of $G$ is at least $3$.

       \item The sums of lengths of edges or the \emph{volume} of $G$ is
       $n$. 
    
    \end{itemize}
  
    We call \X the \emph{moduli space of metric graphs}. It is also
    naturally the quotient of \emph{Outer Space} by the group of outer
    automorphisms of $\F$ (see \cite{vogtmann:OS}). We equip \X with the
    Lipschitz metric: for any two graphs $G$ and $H$, define 
    \[ 
      \dL(G,H) = \min_f \left\{ \log L(f) \right\}, 
    \] 
    where $f$ is a $L(f)$--Lipschitz map from $G$ to $H$. The thick part of
    \X is the subset \Xthick containing those graphs with no loop shorter
    than \ep. We will show:

    \begin{theorem} \label{Thm:UpperDiamOuter}

       \[ \diam_L \big( \Xthick \big) =  O \left( \log \left(
       \frac{n}{\ep} \right) \right). \]

    \end{theorem}

    Let $G$ and $H$ be two graphs in \Xthick. We will construct a map $G
    \to H$ in four steps. The idea is to interpolate $G$ and $H$ by two
    trivalent graphs, $G'$ and $H'$, on which we can apply the algorithms
    of the previous section. The reader may wish to look at the example in
    \figref{ExampleDiamOuter}. 
    
    \begin{figure}[htp!]
    \begin{center}
    \includegraphics{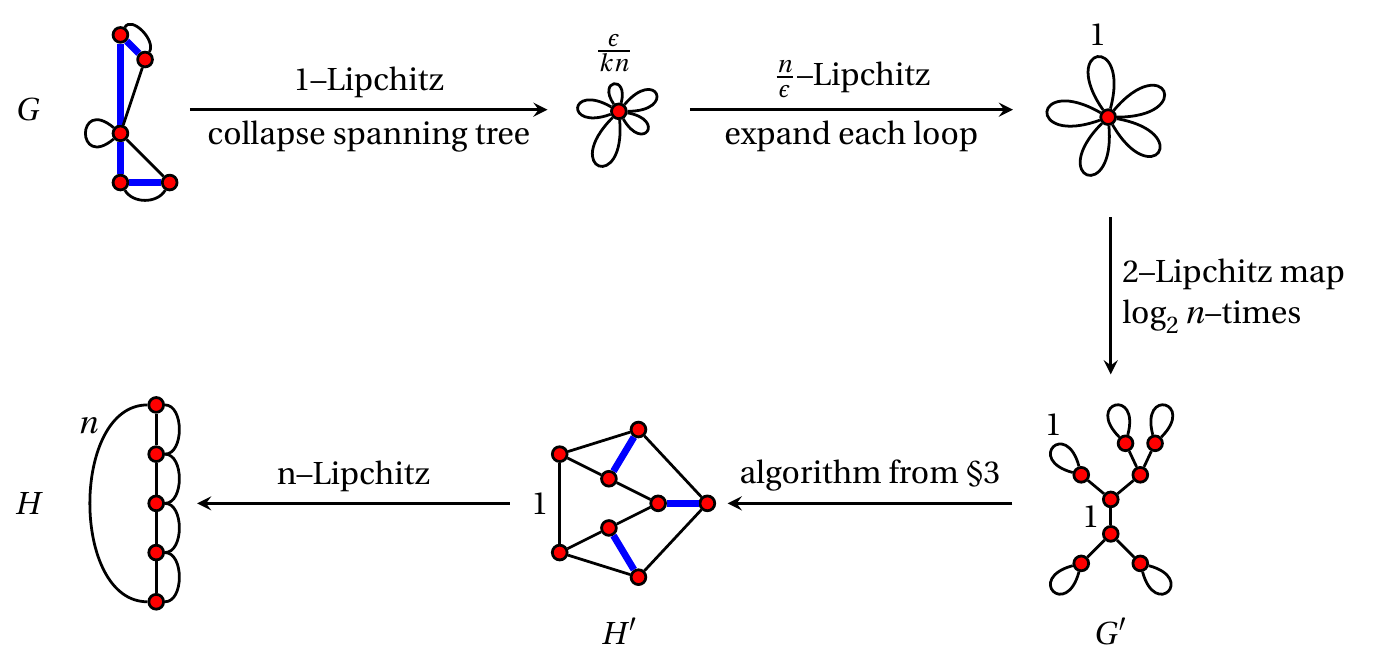}
    \caption{An example in $\calX^\ep_5$.}
    \label{Fig:ExampleDiamOuter}
    \end{center}
    \end{figure}

    We equip $R_n$ with a metric by assigning length $1$ to each circle.

    \begin{lemma}

       The graph $G$ can be mapped to $R_n$ by a $O \left( \frac{n}{\ep}
       \right)$--Lipschitz maps.  

    \end{lemma}

    \begin{proof}
       
       Let $T$ be the shortest spanning tree of $G$. Since $G \in \Xthick$,
       every edge $e$ in the complement of $T$ has length at least
       $\frac{\ep}{kn}$ for some universal $k$. To see this, consider the
       unique loop consisting of $e$ and an embedded path in $T$. Each edge
       in the loop cannot be longer than $e$ since $T$ is the shortest
       spanning tree. Since there are at most $O(n)$ edges in the loop and
       the total length of the loop is at least \ep, this gives the lower
       bound on the length of $e$. Now map $T$ to the vertex of $R_n$ and
       each edge in the complement of $T$ to a circle of $R_n$ via a linear
       map. This map is at most $\frac{kn}{\ep}$--Lipschitz. 
    \end{proof}

    \begin{lemma}

       $R_n$ can be mapped to a trivalent graph $G'$ with all edge lengths
       1 via a composition of $\left\lceil \log_2 n \right \rceil$
       $2$--Lipschitz maps.

    \end{lemma}
    
    \begin{proof}
       
       Divide the circles of $R_n$ into two sets with roughly $n/2$ circles
       each. For each circle, mark off two segments of length $1/4$
       starting from the vertex of $R_n$. For each set, fold all the
       circles together along the marked segments. Note that folding is a
       $1$--Lipschitz map. The resulting graph has an edge $e$ of length
       $1$, and each endpoint of $e$ is attached to roughly $n/2$ loops of
       length $1/2$. Now stretch each loop to have length $1$ by a
       $2$--Lipschitz map and proceed inductively. At each endpoint of $e$,
       divide the circles into two sets of roughly $n/4$ circles each. Then
       fold and stretch. After $\left\lceil \log_2 n \right\rceil$ number
       of steps, we obtain a trivalent graph $G'$ with all edges lengths
       $1$. The composition map $R_n \to G'$ has Lipschitz constant at most
       $2^{\left\lceil \log_2 n \right\rceil} = O(n)$. 
    \end{proof}

    \begin{lemma}

       There is a trivalent graph $H'$ with all edge lengths $1$ that can
       be mapped to $H$ via a $n$--Lipschitz map.     
    
    \end{lemma}
    
    \begin{proof}
       
       For each constant $b$, choose a binary tree $t_b$ with $b$ exterior
       edges. Let $v$ be a vertex of $H$ of valence $b > 3$. Remove a small
       neighborhood of $v$ in $H$ which does not contain any other vertex
       of $H$, and glue the endpoints to the endpoints of $t_b$ in an
       arbitrary way. Now erase the vertices of valence 2 to obtain a
       trivalent graph $H'$. Equip $H'$ with the metric so that each edge
       has length $1$. There is a natural map from $H'$ to $H$ obtained by
       collapsing the edges of $H'$ that are the image of the interior
       edges of $t_b$, and then rescaling the remaining edges of $H'$.
       Collapsing is a $1$--Lipschitz map, and since edges in $H'$ have
       length $1$ and edges in $H$ cannot be longer than $n$, this map is
       $n$--Lipschitz. 
    \end{proof}

    \begin{lemma}
       
       $G'$ can be sent to $H'$ via a composition of $O \big( \log(n)
       \big)$ $L$--Lipschitz maps, where $L$ is a uniform constant.
    
    \end{lemma}
    
    \begin{proof}
       
       There is a uniform constant $L$ such that, for any two trivalent
       graphs with edge lengths 1, if they differ by one simultaneous
       Whitehead move, then they differ by a Lipschitz map with Lipschitz
       constant at most $L$. The graphs $G'$ and $H'$ can be cut into
       binary trees of complexity $2n-2$. By \eqnref{GraphAboveTree} and
       \thmref{UpperTree}, $G'$ can be transformed into $H'$ by $O \big(
       \log (n) \big)$ simultaneous Whitehead moves, hence the statement. 
    \end{proof}
    
    \begin{proof}[Proof of \thmref{UpperDiamOuter}].
      
       We construct a Lipschitz map $G \to H$ as a composition of the maps
       coming from the four Lemmas above:
       \[ 
          G 
          \,\, \xrightarrow{\,\,\,\,O \left( \frac{n}{\ep} \right)\,\,\,\,} \,\,
          R_n 
          \,\, \xrightarrow{\,\,\,O(n)\,\,\,} \,\,
          G' 
          \,\, \xrightarrow{\,L^{O(\log n)}\,} \,\,
          H' 
          \xrightarrow{\quad n \quad}
          H.
       \] 
      The Lipschitz constant of the composition $G \to H$ is bounded by the
      product of the four Lipschitz constants, which is bounded by
      $n^d/\ep$ for some uniform constant $d$. Thus, $\dL(G,H) = O \left(
      \log \left( \frac{n}{\ep} \right) \right)$.
   \end{proof}

\section{Examples of surfaces}

  \label{Sec:Examples}
  
  In this section, we construct some examples of surfaces in the thick part
  of moduli space. These examples will provide the required lower bound for
  the width and the height, and hence the diameter, of the thick part. They
  also showcase some interesting behaviors which are of independent
  interest.

  Let's for the moment restrict our attention to closed surfaces. 
  
  To bound the width from below, we construct three surfaces in \Bers which
  are pairwise $\Omega \big( \log(g) \big)$ apart in the Lipschitz metric.
  These surfaces are constructed using graphs. The first two surfaces, the
  line surface $X$ and the bouquet surface $Y$, are constructed from two
  graphs which have a large ratio between their diameters. This ratio
  computes a lower bound on the Lipschitz constant from $Y$ to $X$. The
  third surface, called the expander surface $Z$, is constructed using an
  \emph{expander graph}, or a graph with high connectivity. We will show
  \thmref{IntroExpander} holds for $Z$: every separating curve on $Z$ is of
  length $\Omega(g)$. This will contrast with $X$ and $Y$, where both
  contains separating curves of length $\ep_M$. Then the length ratio of
  separating curves will provide a lower bound on the Lipschitz distance
  from $X$ or $Y$ to $Z$.

  To bound the height from below, we will construct a surface $H \in
  \Mthick$ that cannot be decomposed into pants by curves shorter than
  $\sqrt{g}$. Buser already has such a construction, called the hairy
  torus, but it does not lie in \Mthick. Our construction essentially takes
  two copies of Buser's hairy torus and glue them along the hairs. The
  resulting surface $H$ has $B(H) = \Omega \big( \sqrt{g} \big)$. Using
  length ratios we obtain a lower bound on the Lipschitz constant from $H$
  to any surface in $\Bers$. 
  
  Most of our constructions generalize easily to punctured surfaces. The
  only exception is the expander surface, as the notion of expanders does
  not exist for graphs in \Graphp, $p>0$, so this example is skipped. For
  the height, we give a construction of that works for all $p \ge 0$ and
  all $g \ge 1$, by combining a double hairy torus and a punctured torus.
  For the remaining case of genus $0$, we will refer to the construction in
  \cite{parlier:BPS} of a hairy sphere.
  
  Since \PM covers \Mp, the height and width of \PM are bounded below by
  the height and width of \Mp.
  
  \subsection{Shadow map} 

  \label{Sec:ShadowMap}
  
  Let $\psi : \PBers \to \PGraph$ be the dual graph map. We will regard
  elements in \PGraph as metric graphs by assigning length 1 to each edge.
  Let $X \in \PBers$. Outside of the cusps, $X$ is quasi-isometric to
  $\psi(X)$. To make this precise, we introduce the \emph{shadow map}
  $\Upsilon: X \to \psi(X)$
  
  For each puncture $p$ of $X$, let $N_p$ be the horocyclic neighborhood of
  $p$ such that the horocyclic boundary of $N_p$ has length equal to
  $\ep_M$. Let $\overline{X}$ be the closure of $X \setminus \bigcup N_p$,
  where $p$ ranges over all punctures of $X$. We call $\overline{X}$ the
  \emph{truncated surface} obtained from $X$.
  
  Let $P$ be the associated pants decomposition on $X$. We may assume
  $\ep_M$ is small enough so that $P$ is contained in $\overline{X}$ and
  that the distance from every curve in $P$ to the boundaries of
  $\overline{X}$ is of order $1$. 
  
  Given a constant $A$, for any $\alpha \in P$, let  
  \[
     N_\alpha = \big\{ \, x \in X \,\,\, : \,\,\, d_X(x,\alpha) \le
     A \, \big\}, 
  \] 
  and for each boundary component $\gamma \in \partial \overline X$, let 
  \[
     N_\gamma = \big\{ \, x \in \overline X \,\,\, : \,\,\, d_X(x, \gamma)
     \le A \, \big\}. 
  \] 
  Choose $A$ so that $N = \displaystyle \bigcup_{\alpha \in P} N_\alpha
  \cup \bigcup_{\gamma \in \partial \overline X} N_\gamma $ is a disjoint
  union of embedded annuli. Each component of $\overline X \setminus N$ is
  a pair of pants with diameter bounded uniformly by a constant $D$.
  Foliate each $N_\alpha$ and $N_\gamma$ by closed loops equidistant from
  $\alpha$ and $\gamma$ respectively. The shadow map 
  \[
     \Upsilon : X \to \psi(X)
  \]
  sends each component in $X \setminus N$ to a vertex and each $N_\alpha$
  or $N_\gamma$ to an edge by collapsing leaves and then mapping linearly
  onto the edge. The map $\Upsilon$ is essentially distance decreasing. For
  any $x, y \in \overline X$, 
  \[
     d_{\psi(X)} \big( \Upsilon(x), \Upsilon(y) \big) \prec d_X(x,y).
  \]
  On the other hand, 
  \[
     d_X(x,y) \le (A+D) \, \left( d_{\psi(X)} \big( \Upsilon(x), \Upsilon(y)
     \big) + 2 \right).
  \]
  Thus, $\Upsilon$ is an quasi-isometry from $\overline X$ to $\psi(X)$.
   
  \subsection{Line and bouquet surfaces} 
  
  \label{Sec:LineBouquet}
  
  We construct two graphs in \Graphp, one of which has diameter $g+p$ and
  the other has diameter $\log(g+p)$. 
  
  Consider the tree $T$ with $g+p$ exterior edges as in \figref{LineGraph}.
  We can make $T$ into an element $\Gamma$ in \Graphp by attaching $g$
  loops to $g$ of those edges. The diameter of $\Gamma$ is at least $g+p$.
  
  Now consider any tree $T'$ of height $\log_2(g+p)$ with $g+p$ exterior
  edges. For instance, one can pick the fully-sorted tree (see
  \figref{FullySortedTree}). Again, $T'$ can be made into a graph $\Gamma'
  \in \Graphp$ by attaching $g$ loops. The graph $\Gamma'$ has diameter at
  most $\log_2(g+p)+2$.
  
  \begin{figure}[!htp] 
  \begin{center}
  \includegraphics{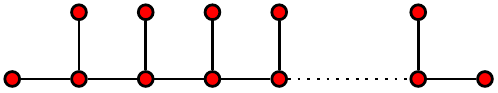}
  \end{center}
  \caption{The tree $T$ giving rise to the line surface.} \label{Fig:LineGraph} 
  \end{figure}

  Let $X=X_{g,p}$ and $Y=Y_{g,p}$ be elements of $\Bersp$ such that
  $\psi(X) = \Gamma$ and $\psi(Y)=\Gamma'$. We will refer to $X$ and $Y$ as
  a \emph{line surface} and \emph{bouquet surface}, respectively. 

  \begin{lemma} \label{Lem:LowerWidth}
  
     \[ \diam_L \big( \Bersp \big) = \Omega \big( \log(g+p) \big). \]
  
  \end{lemma}
   
  \begin{proof}
     
     We will use $\dL(Y,X)$ to achieve this lower bound, where $X$ is the
     line surface and $Y$ is the bouquet surface in $\Bersp$.

     The section on shadow map implies \[ \diam(\overline{X}) \succ
     g+p \qquad \text{and} \qquad \diam(\overline{Y}) \prec \log(g+p), \]
     where $\overline{X}$ and $\overline{Y}$ are truncated surfaces
     obtained from $X$ and $Y$ respectively. 

     Consider any $L$--Lipschitz map $f : Y \to X$. We can choose two
     points $x_1$ and $x_2$ in $\overline{X}$ a distance $\Omega(g+p)$
     apart. Let $y_i=f^{-1}(x_i)$. Since $x_i$ has injectivity radius at
     least $\ep_0$, the injectivity radius at $y_i$ is at least $\ep_0/L$.
     So $y_i$ has distance at most $\log L$ from $\overline Y$. We can
     connect $y_1$ to $y_2$ by an arc $\omega$ in $Y$ with \[
     \ell_Y(\omega) \prec \log(g+p) + 2 \log L. \] The image $f(\omega)$ is
     an arc connecting $x_1$ to $x_2$, so $\ell_X \big( f(\omega) \big)
     \succ (g+p)$. We have
     \[ 
        L \ge \frac{\ell_X \big( f(\omega) \big)}{\ell_Y 
        \big( \omega \big)} \succ \frac{g+p}{\log(g+p) + 2\log L}.  
     \] 
     In the case that $L \ge g+p$, then $\dL(Y,X) = \Omega \big( \log(g+p)
     \big)$. In the case that $L \le g+p$, then the above becomes \[ L
     \succ \frac{g+p}{3 \log(g+p)}.\] 
     Thus, we also obtain 
     \[ 
        \dL(Y, X) \succ \log \left( \frac{g+p}{3\log(g+p)} \right) 
        = \Omega \big( \log(g+p) \big). \qedhere
     \] 
  
  \end{proof}
  
  Combining \lemref{LowerWidth} with \corref{UpperWidth}, we obtain

  \begin{corollary} \label{Cor:DiamWidth}

     The \Teich and the Lipschitz width of \PMthick and \Mpthick are all of
     order $\log(g+p)$.
  
  \end{corollary}
  
  Combining \lemref{LowerWidth} with \corref{WidthAboveGraph} and
  \corref{UpperGraph}, we also obtain:

  \begin{corollary}\label{Cor:DiamGraph}
    
     \[ {\rm \diam_S \big( Tree}(g+p) \big) \asymp  \diam_S \big( \PGraph
     \big) \asymp \diam_S \big( \Graphp \big) \asymp \log(g+p). \]
  
  \end{corollary}
  
  Finally, we can derive the lower bound to the diameter of \Xthick.
  Together with the upper bound coming from \thmref{UpperDiamOuter}, we
  obtain

  \begin{lemma} \label{Lem:DiamOuter}
     
     \[ \diam_L(\Xthick) \asymp \log \left( \frac{n}{\ep} \right). \]

  \end{lemma}
  
  \begin{proof}
      
     We show the lower bound.
     
     Let $\Gamma \in \Xthick$ be a graph of diameter of order $n$. For
     instance, pick $\Gamma$ to be the graph inducing the line surface
     $X_n$, but renormalized to have volume $n$. The wedge $R_n$ of $n$
     circles (with edge lengths 1) has diameter $1$. Thus $\dL (R_n, G)
     \succ \log(n)$. On the other hand, let $H \in \Xthick$ be any graph
     which has a loop of length $\ep$. Then $\dL(H,R_n) \ge \log \left(
     \frac{1}{\ep} \right)$. It follows that \[ \diam_L \big( \Xthick \big)
     \succ \frac{1}{2} \left( \log(n) + \log \left( \frac{1}{\ep} \right)
     \right) = \frac{1}{2} \log \left( \frac{n}{\ep} \right). \qedhere \]
  \end{proof}

  \subsection{Expander surfaces} \label{Sec:Expander}
    
  In this section, we consider only closed surfaces. In this case, the dual
  graph to a pants decomposition is a trivalent graph. We will use
  trivalent graphs with ``high'' connectivity to construct surfaces in
  \Bersp with long separating curves. These surfaces will be $\Omega \big(
  \log(g) \big)$ away from the examples of the previous section, providing
  another proof of \lemref{LowerWidth}. To formalize the notion of
  connectivity, we define the Cheeger constant of a graph. 
  
  Let $\Gamma$ be any graph with $n$ edges. For any subgraph $\Delta$ in
  $\Gamma$, let $\big| \Delta \big|$ be the number of edges in $\Delta$. We
  will let $\partial \Delta \subset \Delta$ be the subset of edges in
  $\Delta$ which share a vertex with an edge outside of $\Delta$. The
  \emph{Cheeger constant} of $\Gamma$ is defined to be
  \[ \ch(\Gamma) = \min_{1 \le |\Delta| \le \frac{n}{2}} \frac{|\partial
  \Delta|}{|\Delta|}, \] where the minimum is taken over all subgraphs
  $\Delta$ with at most $\frac{n}{2}$ edges. 

  \begin{definition} \label{Def:Expander}

     An infinite family $\calE$ of $d$-regular graphs is a
     $\delta$-expander family if $\ch(E) \ge \delta$ for every $E \in
     \calE$.  
  
  \end{definition}

  \begin{theorem}[\cite{pinsker,margulis:ECE}]

     For every $d$, there exists a $\delta_d$--expander family of
     $d$--regular graphs.
     
  \end{theorem}
  
  Fix $\delta$ and let $\calE$ be a $\delta$--expander family of trivalent
  graphs. We will call a surface $Z_g \in \Bers$ an \emph{expander surface}
  if its dual graph $E_g$ is an element $\calE$. By a \emph{dividing curve}
  on a closed surface \s, we will mean a separating curve on \s which
  divides \s into two pieces each of which has genus on the order of $g$.
  We will prove the following fact about expander surfaces, which was known
  to Buser (\cite{buser:Eigenvalues}).

  \begin{theorem} \label{Thm:Expander}
    
     If $Z_g \in \Bers$ is an expander surface, then the shortest dividing
     curve on $Z_g$ has length $\Omega(g)$.

  \end{theorem}
  
  \begin{proof}

     Let $\alpha$ be any dividing curve on $Z_g$, and let $U$ the closure
     of one of the components of $Z_g \setminus \alpha$. By assumption, the
     genus of $U$ is of order $g$.  
     
     Let $\Upsilon : Z_g \to \Gamma_g$ be the shadow map. We claim: 
     \begin{equation} \label{Eq:Expander}
        \big| \Upsilon(U) \big| = \Omega(g) \qquad 
        \text{and} \qquad
        4 \big|\partial \Upsilon(U) \big| \le \big| \Upsilon(\alpha) \big|.
     \end{equation}
     To see the first statement, let $U'$ be the largest subsurface of
     $Z_g$ with the same shadow as $U$ (that is, the preimage of
     $\Upsilon(U)$). Since $U$ is a subsurface of $U'$, we have $\big|
     \chi(U') \big| \geq \big|\chi(U) \big|= \Omega(g)$. But $U'$ is a
     union of pairs of pants, exactly one associated to a vertex in
     $\Upsilon (U)$. Hence $\big|\chi(U')\big| \asymp  \big| \Upsilon(U)
     \big|$. That is, $ \big| \Upsilon(U) \big| = \Omega(g)$. 
     So the first statement follows.
     
     For the second statement, it is sufficient to construct a map from
     $\partial \Upsilon(U)$ to $\Upsilon(\alpha)$ so that the preimage of
     every edge has a uniformly bounded size. Let $e \in \partial
     \Upsilon(U)$. Then $e$ shares a vertex $v$ with an edge outside of
     $\Upsilon(U)$. The vertex $v$ corresponds to a pair of pants and
     $\alpha$ has to intersect this pair of pants non trivially. Hence,
     there is an edge connected to $v$ that is in $\Upsilon(\alpha)$. We
     send $e$ to this edge. The preimage of an edge in $\Upsilon(\alpha)$
     under this map has size at most $4$. 
     
     Using \eqref{Eq:Expander} and the fact that $\calE$ is an expander
     family, we obtain 
     \[ 
     \big| \Upsilon(\alpha) \big| \ge 4\big| \partial \Upsilon(U) \big|
     \ge 4\delta \big| \Upsilon(U) \big| = \Omega(g).
     \] 
     This bounds the length of $\Upsilon(\alpha)$ from below. Since the
     shadow map $\Upsilon$ is essentially distance decreasing, we obtain
     $\ell_{Z_g}(\alpha) = \Omega(g)$. 
  \end{proof}
  
  \begin{lemma} \label{Lem:ExpandertoLineandBouquet}
    
     Let $X_g$, $Y_g$, and $Z_g$ be a line surface, bouquet surface, and
     expander surface, respectively.

     \[ \dL (X_g, Z_g) = \Omega \big( \log(g) \big), \qquad \dL (Y_g,
     Z_g) = \Omega \big( \log(g) \big). \]

  \end{lemma}

  \begin{proof}

     To see $\dL(X_g, Z_g) = \Omega \big( \log(g) \big)$, let $\Gamma_g$ be
     the graph from which we constructed $X_g$. One sees that $\Gamma_g$
     can be divided into two roughly equal-sized pieces by one edge. The
     associated pants decomposition of $X_g$ contains a dividing curve of
     length $\ep_M$.     
     
     On the other hand, from \thmref{Expander}, any dividing curve on $Z_g$
     has length $\Omega(g)$. Any homeomorphism $X_g \to Z_g$ must take
     dividing curves to dividing curves. Using length ratio, we obtain a
     lower bound for the Lipschitz constant.
     \[ 
        \dL(X_g, Z_g) \succ \log \left( \frac{g}{\ep_M} \right)= \Omega
        \big( \log(g) \big).
     \] 
     Similarly, $Y_g$ also admits a dividing curve of length $\ep_M$. So
     the same argument also shows $\dL (Y_g, Z_g) = \Omega \big( \log(g)
     \big)$. 
  \end{proof}

  \subsection{Hairy Torus example and a lower bound for height}

  \label{Sec:HairyTorus} 

  Recall that surfaces in \Bersp can be decomposed into pants by curves of
  length $\ep_M$. In the following, we will construct a surface in \Mpthick
  which cannot be decomposed into pants by curves all shorter than
  $\sqrt{g+p}$.
  
  \begin{lemma}\label{Lem:HairyTorus}

     There exists a surface $H \in \Mpthick$ such that $B(H) = \Omega \big(
     \sqrt{g+p} \big)$

  \end{lemma}

  Assuming \lemref{HairyTorus}, we derive the lower bound for the height of
  the thick part of moduli space:

  \begin{corollary} \label{Cor:LowerHeight}
     
     \[ \HD_L \big(\PMthick, \PBers \big) \ge \HD_L\big(\Mpthick, \Bersp
     \big) = \Omega \left( \log \left( \frac{g+p}{\ep} \right) \right).\]

  \end{corollary}
  
  \begin{proof}
     
     Let $H$ be the surface of \lemref{HairyTorus}. For any $X \in \Bersp$,
     any Lipschitz map from $X$ to $H$ must take a pants decomposition of
     $X$ to a pants decomposition of $H$. Thus some curve on $X$ must get
     stretched by a factor $\Omega \left( \frac{\sqrt{g+p}}{\ep_M}
     \right)$. This means \[ \dL(X,H) \succ \log \left( \frac{g+p}{\ep_M}
     \right). \] Now consider a surface $Y \in \Mpthick$ which has a curve
     of length $\ep \le \ep_M$. Any Lipschitz map from $Y$ to $X$ must
     stretch this curve by a factor at least $\frac{\ep_M}{\ep}$, so \[
     \dL(Y,X) \ge \log \left(\frac{\ep_M}{\ep}\right).\] It follows then
     \[ 
     \HD_L \big( \Mpthick, \Bersp \big) \succ \frac{1}{2} 
     \left( \log \left( \frac{g+p}{\ep_M} \right) + \log \left(
     \frac{\ep_M}{\ep}
     \right) \right) 
     = \frac{1}{2} \log \left( \frac{g+p}{\ep} \right). \qedhere
     \]
  \end{proof}
  
  The rest of this section is dedicated to constructing $H$ for
  \lemref{HairyTorus}. In the case of genus $g=0$, $p>0$, this has already
  been done by Balacheff and Parlier in \cite{parlier:BPS}. Their hairy
  sphere construction gives rise to a surface $H \in
  \calM_{0,p}^{\,\ep}/\Sym_p$ with $B(H) = \Omega \big( \sqrt p \big)$. For
  higher genus, we will use a variation of Buser's hairy torus \cite[\S
  5.3]{buser:GSC}. We will first explain the construction in the case of
  closed surfaces. Then we will extend the construction to punctured
  surfaces with at least one genus.
  
  Start with a right-angled geodesic pentagon in hyperbolic plane with side
  lengths $a, b, c, d, e$. We set $a=b$, $c=e$ and $d=.25$. Glue four
  copies of such pentagons together to form a $2a \times 2a$ hyperbolic
  square $R$ with an inner geodesic boundary component $\gamma$ of length
  $1$ (see \figref{Pentagons}). 

  \begin{figure}[!htp] 
  \begin{center} 
  \includegraphics{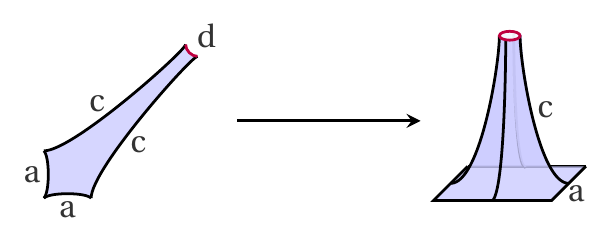}
  \end{center}
  \caption{Gluing pentagons} 
  \label{Fig:Pentagons} 
  \end{figure}

  Given any positive integer $m$, we can paste together $m^2$ copies of $R$
  to obtain a larger square $R_m$ of side lengths $2am$ with $m^2$ inner
  boundary components. We index these boundaries by $\gamma_{ij}$.
  Identifying the opposite sides of $R_m$ yields a hyperbolic surface (a
  hairy torus) $T_m$ of genus $1$ with $m^2$ boundary components. Now take
  two copies of $T_m$ and glue them along the $\gamma_{ij}$'s. The
  resulting closed surface $H_m$ has genus $g=1+m^2$ (see
  \figref{HairyTorus}).
  
  To show $H_m$ satisfies \lemref{HairyTorus}, we define a Lipschitz map
  $\pi: H_m \to F_m$, where $F_m$ is a flat torus obtained from gluing the
  opposite sides of a $2am \times 2am$ Euclidean square. The map $\pi$ is
  defined locally on each hyperbolic square $R$ with an inner boundary
  $\gamma$. 
  
  Let $F$ be a $2a \times 2a$ Euclidean square. Let $\pi_R : R \to F$ be
  any Lipschitz map that takes the sides of $R$ to the sides of $F$ and
  $\gamma$ to the center of $F$. Let $L$ be the Lipschitz constant of
  $\pi_R$. Divide $F_m$ into $m^2$ sub-squares. Let $\pi : H_m \to F_m$ be
  the map such that, on each hyperbolic square $R$ in $H_m$, $\pi$
  restricted to $R$ is mapped to the corresponding sub-square $F$ in $F_m$
  via the map $\pi_R$ (see \figref{HairyTorus}). The Lipschitz constant of
  $\pi$ is at most the Lipschitz constant of $\pi_R$. We have shown the
  following.

  \begin{lemma}\label{Lem:Reducing}

     The map $\pi : H_m \to F_m$ is $L$--Lipschitz for a uniform $L$.

  \end{lemma}
  
  The map $\pi$ has the following property.

  \begin{lemma} \label{Lem:Homology}

     For any pants decomposition $P$ on $H_m$, there exists a curve $\alpha
     \in P$ such that $\pi(\alpha)$ is not trivial in homology. Therefore,
     $H_m$ satisfies \lemref{HairyTorus}.
  
  \end{lemma}

  \begin{proof}
    
     For the purpose of this proof, we can consider $\pi$ up to homotopy.
     The map $\pi : H_m \to F_m$ is onto on homology, since $\pi$ maps each
     hairy torus in $H_m$ onto $F_m$. If a curve $\beta \in P$ maps to a
     trivial curve in the homology of $F_m$, then one can find a disk $D$
     in $F_m$ such that $\pi(\beta)$ lies in $D$. We may assume $\partial
     D$ avoids the singular points (the centers of the sub-squares) of
     $F_m$. The complement $F_m \setminus D$ is a one-holed torus. Its
     preimage $Z = \pi^{-1}(F_m \setminus D)$ in $H_m$ is an essential
     subsurface which is not a pair of pants (since $\pi(Z)$ maps onto the
     homology of $F_m$). Since $P$ is a pants decomposition, some curve
     $\alpha \in P$ must intersect $Z$. The image $\pi(\alpha)$ cannot be
     homotoped away from $F_m \setminus D$, thus it is non-trivial in the
     homology of $F_m$. 
     
     For such an $\alpha$, we have $\ell_{F_m} \big( \pi(\alpha) \big) \ge
     2am = 2a \sqrt{g-1}$. Since $\pi$ is $L$--Lipschitz, $\ell_{H_m}(
     \alpha ) = \Omega(m) = \Omega \big( \sqrt g \big)$. This is true for
     any pants decomposition $P$, so $B(H_m) = \Omega \big( \sqrt g \big)$. 
  \end{proof}
  
  \begin{figure}[!htp]
  \begin{center}
  \includegraphics{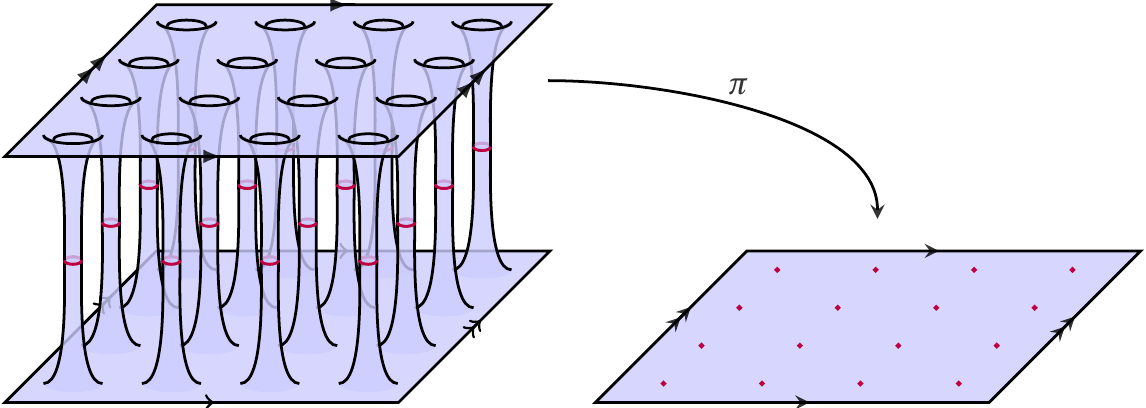}
  \end{center}
  \caption{Hairy torus}
  \label{Fig:HairyTorus}
  \end{figure}

  The double-torus construction yields a closed surface $H_m$ of genus
  $1+m^2$. We now extend the construction to every genus $g$. Given an
  arbitrary $g$, let $m$ be the largest integer such that $1 + m^2 \le g <
  1 + (m+1)^2$. Let $r = g- (1+m^2) < 2m+1$. Let $H_m$ be the closed
  surface of genus $1+m^2$ obtained from gluing two hairy tori together.
  Cut $H_m$ along one of the curves of length $1$ which came from the
  gluing, to obtain a surface of genus $m^2$ with two boundary components.
  For the pair of boundary components $\beta$ and $\beta'$, we insert a
  surface with boundary as follows. Let $W$ be a hyperbolic surface of
  genus $r$ with two geodesic boundary components. We require the boundary
  components of $W$ to have length 1, and all non-peripheral curves on $W$
  to have length at least $\ep$. We glue a copy of $W$ along its boundary
  components to $\beta$ and $\beta'$. The resulting surface $H$ is a closed
  surface of genus $g=m^2+r+1$. In the map previously defined on $H_m \to
  F_m$, the pair $\beta$ and $\beta'$ are mapped to the same point in $F$.
  Thus, this map can be extended to $W$ by a constant map. This defines a
  Lipschitz map $H \to F_m$ with the same constant as the map $H_m \to
  F_m$. The same proof in \lemref{Homology} also works to show $B(H) =
  \Omega(m) = \Omega \big( \sqrt{g} \big)$.
  
  Now we explain the construction in the case of punctured surfaces of
  genus $1$. For any $a$, there exists a hyperbolic quadrilateral with $3$
  right angles and one ideal vertex, such that the two sides opposite of
  the ideal vertex have length $a$. Glue four copies of such quadrilaterals
  together to form a $2a \times 2a$ hyperbolic square $R'$ with a puncture.
  The surface $R'$ can be mapped to $F$ minus the center by a uniformly
  Lipschitz map. Pasting together $m^2$ copies of $R'$ and gluing opposite
  sides yields a surface $Q_m$ of genus $1$ with $p=m^2$ punctures. Gluing
  the maps on each $R'$ equips $Q_m$ with a Lipschitz map $\pi$ to the flat
  torus $F_m$. By the same proof as in \lemref{Homology}, any pants
  decomposition $P$ on $Q_m$ must contain a curve $\alpha$ which does not
  vanish in the homology of $F$. Hence $\ell_F \big( \pi(\alpha) \big) \ge
  2am=2a\sqrt p$, which implies $\ell_{Q_m}(\alpha) = \Omega \big( \sqrt{p}
  \big) = B(Q_m)$. 
  
  For an arbitrary $p$, we do a similar modification as in the case of
  closed surfaces. Let $r=p-m^2<2m+1$. Let $Q_m$ be the genus $1$ surface
  with $m^2$ punctures as constructed above. Now remove an $R'$--square of
  $Q_m$ and replace it with an $R$-square. The result is a surface $Q_m'$
  of genus 1 with $m^2-1$ punctures and one boundary component of length 1.
  One checks that $B(Q_m') \asymp B(Q_m)=\Omega \big( m \big)$. To the
  inner boundary of $R$, attach a sphere with $r+1$ punctures with a
  boundary component of length $1$, such that all essential curves also
  have length at least $\ep$. The result is a surface $Q$ of genus $1$ with
  $p$ punctures. The map $Q \to F_m$ is defined locally in each $R$ and
  $R'$ and extended by the constant map to the attached sphere. The same
  reasoning as before shows $B(Q) = \Omega \big( m \big) = \Omega \big(
  \sqrt p \big)$.   

  Finally, we combine the two constructions. Let $g > 1$ and $p$ be
  arbitrary. Let $H$ be a closed surface of genus $g-1$ obtained from the
  double hairy torus construction. Cut $H$ along one of the curves of
  length $1$ to obtain a surface $H'$ of genus $g-2$ with two boundary
  components. From above, we can construct a surface $Q'$ with genus 1, $p$
  punctures, and one boundary component of length $1$, such that
  $B(Q')=\Omega(\sqrt p)$. Now glue $Q'$ to $H'$ via an intermediate pair
  of pants with three boundary components of lengths 1. The resulting
  surface $X$ has the right topology. The Lipschitz map from $H$ to $F_m$,
  where $m \asymp \sqrt{g}$, sends the boundary components of $H'$ to the
  same point. Therefore, this map extends to all of $X$ by a constant map
  on $Q'$ and the intermediate pants. Likewise, the Lipschitz map from $Q'$
  to $F_{m'}$, $m' \asymp \sqrt{p}$, can be extended to $X$ by a constant
  map on $H'$ and the intermediate pants. In both cases, for any pants
  decomposition $P$ on $X$, the image of $P$ in either $F_m$ or $F_{m'}$ is
  non-trivial in homology. Therefore, $P$ contain a curve $\alpha$ and a
  curve $\alpha'$, such that $\ell_X(\alpha) = \Omega \big( \sqrt g \big)$
  and $\ell_X(\alpha') = \Omega \big( \sqrt p \big)$. It follows then \[
  B(X) = \Omega \big( \sqrt g + \sqrt p \big) = \Omega \big( \sqrt{g+p}
  \big).\] This concludes the proof of \lemref{HairyTorus} and this
  section.

\section{Upper bound for height}

  \label{Sec:UpperHeight}
  
  In this section, we give an upper-bound for the height of \PMthick. For
  any $X \in \PMthick$, we will show that there exists a surface $Y \in
  \PBers$ and a map from $Y$ to $X$ such that, for any curve $\gamma$ on
  $X$, the ratio of $\ext_X(\gamma)$ to $\ext_Y(\gamma)$ is bounded above
  by a polynomial function in $\frac{g+p}{\ep}$.  Then \thmref{Ker} would
  provide an upper bound of $\log \left( \frac{g+p}{\ep} \right)$ for the
  \Teich distance between $X$ and $Y$.
  
  Let $X \in \PMthick$ be given and let $P$ be the shortest pants
  decomposition of $X$. By \thmref{Buser}, there is an upper bound of order
  $(g+p)$ for the lengths of the curve in $P$ (the bound does not depend on
  $\ep$). Let $Y \in \PBers$ be a surface where there is a pants
  decomposition $P'$ in the homeomorphism class of $P$ where the curves
  have lengths $\ep_M$. For $\alpha \in P'$, choose the shortest transverse
  curve $\balpha$ to $\alpha$ which intersects $\alpha$ minimally and is
  disjoint from the other curves in $P'$. All transverse curves have length
  of order $1$. The curves in $P'$ and their duals form a set of curves
  $\mu$ which is usually referred to as a clean marking on $Y$.   

  There is map from $Y$ to $X$ that sends the curves in $P'$ to curves in
  $P$. In fact, the homotopy class of this map is unique up to Dehn twist
  around curves in $P$. We choose such a map $f$ so that, for any $\alpha
  \in P'$, the length of $f(\balpha)$ in $X$ has an upper bound of order
  $M=\max \left\{ g+p, \log \left( \frac{1}{\ep} \right) \right\}$. To see
  that such map exists, let $\beta$ be any curve in $P$. We will find a
  transverse curve $\bbeta$ to $\beta$ of length $O(M)$. Cutting $X$ along
  curves all curves in $P$ except $\beta$ leaves a subsurface containing
  $\beta$ which is either a torus with one boundary component or a sphere
  with four boundary components. The length of each boundary component has
  an upper bound of order $g+p$ and a lower bound of $\ep$. Consider the
  case that $\beta$ is contained in a torus with a boundary component
  $\beta'$. There are a pair of arcs $\omega$ and $\omega'$ in the torus
  that are perpendicular to $\beta$ and $\beta'$. One can use elementary
  hyperbolic geometry to show that the lengths of $\omega$ and $\omega'$
  are $O(M)$. By a surgery using $\omega$ and $\omega'$ and arcs in
  $\beta'$ and $\beta$, we can construct a curve that intersects $\beta$
  exactly once. The length of geodesic representative $\bbeta$ of the curve
  is bounded above by the sum of the lengths of $\beta$, $\beta'$, $\omega$
  and $\omega'$, thus $\ell_X(\bbeta) = O(M)$. A similar construction also
  works in the case that $\beta$ is contained in a sphere with four
  boundary components. 
  
  Hence, we do not distinguish between $P$ and $P'$ and denote them both by
  $P$ and we can consider $\mu$ as a marking on $X$. 
   
  Let $\gamma$ be a curve in $Y$. First we compare the extremal length and
  the hyperbolic length of $\gamma$ in $Y$. Recall the definition of
  extremal length from the background section. If we pick $\rho$ to be the
  hyperbolic metric on $Y$, we have the following inequality just from the
  definition:
  \begin{equation} \label{Eq:eY>lY}
    \ext_Y(\gamma) \geq \frac {\ell_{\rho}(\gamma)^2}{\area(\rho)} =
    \frac{\ell_Y(\gamma)^2}{2\pi\chi(Y)} \succ \frac{\ell_Y(\gamma)^2}{g+p} . 
  \end{equation}

  We now estimate the $\ell_Y(\gamma)$ using its intersection pattern with
  curves in $\mu$: Recall that curves in $\mu$ are either pants curve
  coming from $P$, usually denoted by $\alpha$, or dual curves, denoted by
  $\balpha$. For simplicity, when we write $\alpha \in \mu$, we allow
  $\alpha$ to be both a pants curve or a dual curve. In the case where
  $\alpha$ is a pants curve, $\balpha$ would be the pants curve that
  $\alpha$ is dual to (every curve is the dual of the dual). 
  \begin{lemma}
     
     \begin{equation} \label{Eq:lY>intY}
        \ell_Y(\gamma) \succ 
        \sum_{\alpha \in \mu} i(\gamma,\alpha)\ell_Y(\balpha).
    \end{equation} 

  \end{lemma}
  
  \begin{proof}
     
     This is well known. Essentially, every time $\gamma$ intersects
     $\alpha$, it has to cross an annulus with thickness comparable to
     the length of $\balpha$. 
  \end{proof}
  
  Note that a curve $\alpha \in \mu$ has length of order $1$ in $Y$.
  Also, 
  \[
     \ell_X(\alpha) \prec M. 
  \]
  Hence, we can control how the sum on the right-hand side of \eqnref{lY>intY} 
  changes when we replace $Y$ with $X$. That is, 
  \begin{equation} \label{Eq:intY>intX}
   \sum_{\alpha \in \mu} i(\gamma,\alpha) \ell_Y(\balpha) \succ \frac{1}{M}
   \sum_{\alpha \in \mu} i(\gamma,\alpha) \ell_X(\balpha).
  \end{equation}

  This sum, in turn, provides an upper bound for the length of $\gamma$
  in $X$. The following formula was proved in \cite[Proposition 3.2]{tao:BC}

  \begin{lemma}
    
    For any curve $\gamma$ on $X$,

     \begin{equation} \label{Eq:intX>lX}
        2 \sum_{\alpha \in \mu} i(\gamma,\alpha) \ell_X(\balpha) \geq
        \ell_X(\gamma). 
     \end{equation}

  \end{lemma}

  The final step is to compare the hyperbolic length and the extremal
  length of $\gamma$ in $X$. 
  
  \begin{lemma} \label{Lem:lX>eX}
   
     For any $X \in \PMthick$ and any curve $\gamma$, 
     
     \begin{equation} \label{Eq:lX>eX}
       \ell_X(\gamma)^2 \succ \frac {\ep^2}{g+p} \ext_X(\gamma). 
    \end{equation} 

  \end{lemma}
  
  To prove this, we use the following lemma which essentially follows
  from the definition of the extremal length and is a special case of 
  \cite[Lemma 4.1]{minsky:PR}:
  
  \begin{lemma}
  Given a metric $\sigma$ on $S$ and a representative $\hat \gamma$ of a
  simple closed curve $\gamma$, let $v= v(\sigma, \hat \gamma)$ be a number
  so that the $v$--neighborhood of $\hat \gamma$ is homeomorphic to a
  standard product $\hat \gamma \times [0,1]$. Then 
  \[
  \ext_\sigma(\gamma) \leq \frac {\area(\sigma)}{4v^2}.
  \]
  \end{lemma}

  \begin{proof}[Proof of \lemref{lX>eX}]

     Let $X$ be a point in $\Mthick$. We would like to apply the above
     lemma. However, when $\ell_X(\gamma)$ is large, the geodesic
     representative of $\gamma$ gets exponentially close to itself, that is
     $v(\gamma, X) \asymp e^{-\ell_X(\gamma)}$. Hence, to obtain a
     polynomial bound, we need to perturb $\gamma$ and push it away from
     itself as much as possible. Our approach is to triangulate the surface
     and spread $\gamma$ locally in each triangle.
          
     Let $\overline{X}$ be the truncated surface obtained from $X$ (see
     \secref{ShadowMap}). We may assume all curves of $P$ are contained in
     $\overline{X}$. Choose a triangulation $T$ on $\overline{X}$ so that
     the length of any edge in $T$ has an upper bound of $\ep/2$ and a
     lower bound of order $\ep$ and so that there is a uniform lower bound
     on angles of every triangle in $T$. Such $T$ can be constructed using
     a Delaunay triangulation on random well-spaced points on
     $\overline{X}$, which has been constructed explicitly in
     \cite{leibon:DT}. The edge lengths of the triangles of $T$ take values
     in the interval $\left[ \frac{\ep}{k}, \frac{\ep}{2} \right]$ for a
     uniform $k$. The triangles in this construction have bounded
     circumradii on the order of $\ep$. This fact provides a uniform lower
     bound for the angles (see \cite{breslin:TT}). 
     
     We may also perturb the triangulation so that $\gamma$ does not pass
     through any vertex of $T$ and is not tangent to any edge of $T$. 

     \begin{claim}

        For any edge $e$ in $T$, \[\left| \gamma \bigcap e \right| \leq
        \frac {2 \ell_X(\gamma)}\ep.\]

     \end{claim}
     
     \begin{proof}[Proof of the claim]

        Let $p,q \in \gamma \bigcap e$ be two intersection points that
        appear consecutively along $\gamma$. Consider the loop formed by
        taking the union of the arc $\omega_\gamma \subset \gamma$
        connecting $p$ to $q$ with the segment $\omega_e \subset e$
        connecting $p$ to $q$. This loop is essential in $X$ and thus must
        have length at least \ep. The length of the arc $\omega_e$ is at
        most $\ep/2$. Thus the length of arc $\omega_\gamma$ is at least
        $\ep/2$. Hence, the intersection number is less than
        $2\ell_X(\gamma)/\ep$.
     \end{proof}

     \begin{figure}[!htp]   
     \begin{center}
     \includegraphics{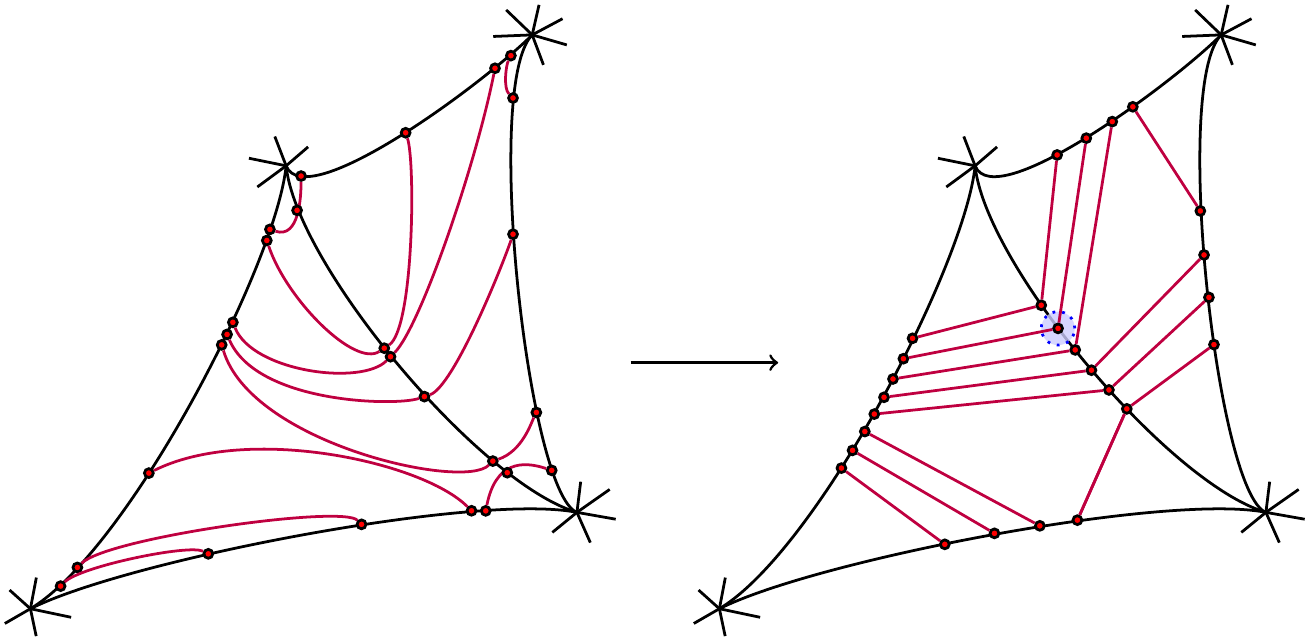}
     \end{center} 
     \caption{Perturbing arcs of $P$ in $\Delta$.} 
     \label{Fig:Perturb} 
     \end{figure}
     
     Let $l = \ell_X(\gamma)$. We refer to \figref{Perturb} in the
     following construction. For a triangle $\Delta \in T$, the restriction
     of $\gamma$ to $\Delta$ is a collection of $O\big(l/\ep\big)$ arcs.
     For every edge $e$ of $\Delta$, let $i_e =\left| \gamma \bigcap e
     \right|$. When $i_e >1$, let $\bar e$ be the middle third of the
     segment $e$. Mark $i_e$ points on $\bar e$ subdividing it into
     $(i_e-1)$ equal segments. When $i_e=1$, $\bar e$ is just the middle
     point of $e$ and it contains one marked point. When $i_e=0$, $\bar e$
     is empty. We replace the restriction of $\gamma$ to $\Delta$ with a
     collection of straight segments that start and end with the chosen
     marked points. Note that the distance between segments is at least of
     order $\ep/l$. We denote the resulting curve by $\hat \gamma$ which is
     an slight perturbation of $\gamma$ and hence is homotopic to it. For
     every point in $\hat \gamma$, there is a neighborhood with radius of
     order $\ep$ that is contained in the union of at most two triangles.
     Thus, there is a neighborhood of $\hat \gamma$, with the thickness of
     order $\ep/l$, that is standard. That is, $$ v=v(\hat \gamma, X) \succ
     \frac \ep l. $$ Taking $\sigma$ to be the hyperbolic metric on $X$
     and using Minsky's Lemma we obtain
     \[ 
        \ext_X(\gamma) \prec \frac{\area(X)}{\ep^2/l^2}.
    \]
    The proof follow from the fact that the area of $X$ is of order
    $(g+p)$.
    \end{proof}
   
    We can now combine these results to obtain the desired upper bound for
    the height of the thick part.
   
    \begin{theorem} \label{Thm:UpperHeight}
       
       \[ \HD_T \big( \PMthick, \PBers \big) = O \left( \log \left(
       \frac{g+p}{\ep} \right) \right).
       \]

    \end{theorem}
   
    \begin{proof}
    
    We will show that for any $X \in \PMthick$ there is $Y \in \PBers$ so
    that
       \[
       \dT(Y,X) = O \left( \log \left( \frac {g+p}\ep \right) \right).
       \]
    Let $\gamma$ be a curve on $X$. We can can combine Equations
    \eqref{Eq:eY>lY},   \eqref{Eq:lY>intY},   \eqref{Eq:intY>intX} 
     \eqref{Eq:intX>lX},   \eqref{Eq:lX>eX}
     ( multiply \eqref{Eq:eY>lY},  \eqref{Eq:lX>eX} and the square of   
     \eqref{Eq:lY>intY},   \eqref{Eq:intY>intX} and \eqref{Eq:intX>lX})
    to obtain:
    \[
     \ext_Y(\gamma) \succ \left( \frac \ep{g+p} \right)^2 \frac{1}{M^2}
     \ext_X(\gamma),
    \] 
    where $M = \max \left\{g+p, \log \left( \frac{1}{\ep} \right\} \right)$. We overestimate 
    an upper bound for $M$ for uniformity: $M \prec \frac{g+p}{\ep}$. Hence, after 
    reorganizing, we have 
    \[
      \frac{\ext_X(\gamma)}{\ext_Y(\gamma)} \prec \left( \frac {g+p}\ep
      \right)^2 M^2 \prec \left( \frac {g+p}\ep
      \right)^4. 
    \]
    Since this is true for
    every curve $\gamma$, applying \thmref{Ker}, we obtain the desired
    upper bound
    \[
       \dT(X,Y) \le \frac 12 \sup_\gamma
       \log \frac{\ext_X(\gamma)}{\ext_Y(\gamma)} \prec 2 \log \left(\frac
       {g+p}\ep \right). \qedhere
    \]
    \end{proof}
   
  Combining \thmref{UpperHeight} with \corref{LowerHeight}, we deduce that

  \begin{corollary} \label{Cor:Height}
     
     The \Teich and the Lipschitz height of \Mpthick and \PMthick are all of
     order $\log \left( \frac{g+p}{\ep} \right)$.

  \end{corollary}
  
  Finally, combining \corref{Height} with \corref{DiamWidth} and the
  triangle inequality, we obtain

  \begin{corollary}
     
     The \Teich and the Lipschitz diameter of \Mpthick and \PMthick are all
     of order \\ $\log \left( \frac{g+p}{\ep} \right)$.

  \end{corollary}
  
  \begin{proof}
    
    Note that the lower bound for the Lipschitz height of \Mpthick is also
    a lower bound for the diameter of \Mpthick. Thus, $\diam_L(\Mpthick) =
    \Omega \left( \log \left( \frac{g+p}{\ep} \right) \right)$. What
    remains to show is the upper bound for the \Teich diameter of \PMthick.

    Let $X, Y$ be any two points in $\PMthick$. By \thmref{UpperHeight},
    there exist $X'$ and $Y'$ in $\PBers$, such that \[ \dT(X,X') \prec
    \log \left( \frac{g+p}{\ep} \right) \qquad \dT(Y,Y') \prec \log \left(
    \frac{g+p}{\ep} \right). \] By \corref{DiamWidth}, we have
    \[ \dT(X',Y') \prec \log \left( \frac{g+p}{\ep} \right). \] Thus, by
    the triangle inequality, (note that \Teich metric is symmetric), \[
    \dT(X,Y) \le \dT(X,X') + \dT(X',Y') + \dT(Y',Y) \prec 3 \log \left(
    \frac{g+p}{\ep} \right). \qedhere\]

  \end{proof}

\appendix

\section{The metric of Whitehead moves on the space of graphs}
  
  \label{Sec:Appendix}

  In this section, we compute the asymptotic diameter of \PGraph and
  \Graphp in the metric of Whitehead moves. Recall that this is equivalent
  to computing the asymptotic diameter of the moduli space of pants
  decompositions on \sp and $\sp/\Sym_p$ in the metric of elementary moves.
  The main results and proofs are extrapolated from known results. Our
  purpose for writing this section is to unify what is known and to put it
  into our context. This section can be read independently from the rest of
  the paper. 
  
  Let $d_W$ represent the metric of Whitehead moves: two graphs have
  distance one if they differ by one Whitehead move. The main results we
  would like to present are as follows. Note that, as a matter of
  convention, we have $x \log(x) =0$, when $x=0$.

  \begin{theorem}[Labeled punctures] \label{Thm:WHLabeled}
     
     \[ \diam_W \big( \PGraph \big) \asymp (g+p) \log(g+p).\]
     
  \end{theorem}

  \begin{theorem}[Unlabeled punctures] \label{Thm:WHUnlabeled}
     
    \[ \diam_W \big( \Graphp \big) \asymp g \log(g+p)+(g+p).\]
     
  \end{theorem}
  
  Our arguments for the two theorems are based on the work of
  \cite{boll:NRG} and \cite{thurston:RDT}. In the case $p=0$, we also refer
  to \cite{cavendish:WH} for an alternate proof of $\diam_W \big( \Graph
  \big) \asymp g \log(g)$.
  
  We first argue for the upper bounds. The upper-bound in \thmref{WHLabeled}
  follows easily from \thmref{IntroDiamGraph}. Namely, since a simultaneous 
  Whitehead move is a composition of at most $g+p$ Whitehead moves, we have
  \[ \diam_W \big( \PGraph \big) \prec (g+p)\log(g+p). \] Similarly, if $g
  \geq p$, then the same argument also shows \[ \diam_W \big( \Graphp)
  \prec (g+p) \log(g+p) \prec g \log(g+p).\] This argument fails for
  $p>g$, so we present another one that works in general. 

  When $g=0$ and $p \ge 2$, then $\text{Graph}(0,p)/\Sym_p$ is the space of 
  unlabeled (unrooted) trees with $p$ ends. This case has been dealt with in 
  \cite{thurston:RDT} where it is shown that 
  \begin{equation} \label{Eq:Flip}
  \diam_W(\text{Graph}(0,p)/\Sym_p) \asymp p.
  \end{equation}
  (In fact, their estimate is very explicit with only a small additive error.) 

  Now suppose $g \ge 1$. We use induction on $g$ and assume
  \thmref{WHUnlabeled} for all smaller values of $g$. Let $\Gamma_0 \in \Graphp$
  be a graph which has a loop of length $1$ (an edge $e$ starting and ending
  at the same vertex). Let $\Gamma_1 \in \text{Graph}(g-1, p+1)/\Sym_{p+1}$
  be the graph obtained from $\Gamma_0$ after removing $e$. 
  Let $\Gamma$ be any other graph in $\Graphp$.  It is
  sufficient to show that $\Gamma$ is distance $O\big(\log(g+p)\big)$ from a 
  graph $\Gamma'$ with a loop $e'$ of length one. This is because
  $\Gamma' \setminus \{e'\}$ can be transformed to $\Gamma_1$ in
  \[
  O\Big( (g-1) \log\big( (g-1) + (p+1)\big) + \big((g-1) + (p+1)\big). 
  \Big)
  \]
  many steps using induction and the theorem follows.
  
  Let $T$ be a spanning subtree of $\Gamma$. Then $T$ has
  $(g+p)$ ends. Using \eqnref{Flip}, we can use Whitehead moves with support 
  in $T$ to transform  $\Gamma$ to a graph with a spanning tree of diameter 
  $\log(g+p)$. We denote the new graph by $\Gamma$ as well. 
  Then $\Gamma$ has a loop of length of order $\log(g+p)$. 
  That is, for $n \prec \log(g+p)$, there is a linear map $\phi \colon [0,n] \to \Gamma$ 
  such that the image is a non-trivial loop, $\phi(0) = \phi(n)$, and $\phi \big( [i-1,i] \big)$
   is an edge of $\Gamma$.  One can always apply a Whitehead move at an edge
  on the loop to shorten the length of the loop by at least one. Hence, after
  $O\big(\log(g+p)\big)$--many moves, there is a loop of length one.
  This proves the upper bound. 

  The lower bounds for both theorems are obtained by a counting argument.
  We first need an asymptotic formula for the cardinalities of $\PGraph$
  and $\Graphp$. For the following, we introduce the notation $A \sim B$ to
  mean $A \asymp c^{g+p} B$, where $c$ is an uniform constant. Since we
  will be applying the logarithm later, exponential factors can be ignored.
  
  By the work of \cite{boll:NRG}, we have \[ \big| \PGraph \big|
  \sim \frac{(6g+2p)!}{g!(2g+p)!}\] 
  Up to exponential factors, this simplifies to 
  \begin{equation} \label{Eq:CardLabeled}
    \big| \PGraph \big| \sim \frac{(g+p)^{6g+2p}}{g^g(g+p)^{2g+p}} \sim
  (g+p)^{g+p}. 
  \end{equation}
  Similarly, we have 
  \begin{equation} \label{Eq:CardUnlabeled}
    \big| \Graphp \big| \sim \frac{(6g+2p)!}{g!p!(2g+p)!} \sim
  \frac{(g+p)^{4g+p}}{g^g p^p}.
  \end{equation}

  To finish the proof, we need the result in \cite{thurston:RDT}, which
  gives an upper-bound of the form $c^{g+p+r}$ for the cardinality of a
  ball of radius $r$ in \PGraph and in \Graphp, where $c$ is some fixed
  constant (See \cite[Theorem 2.3]{thurston:RDT}). (In fact, their theorem
  is much more general and applies to any space of shapes when shapes are
  allowed to evolve through locally supported elementary moves, such as 
  Whitehead moves.)
  
  Let $r$ be the minimal number such that the ball of radius $r$ contains
  the whole space \PGraph. By \eqnref{CardLabeled}, we have
  \[ 
    c^{g+p+r} \succ  (g+p)^{g+p}.
  \]
  Taking the logarithm of both sides, we obtain 
   \[
   g+p+r \succ  (g+p) \log(g+p) 
   \quad \Longrightarrow \quad 
    \diam_W \big( \PGraph \big) \geq r \succ (g+p) \log(g+p).
  \]
  This completes the proof of \thmref{WHLabeled}.
  
  For unlabeled punctures, we prove two lower bounds. Since our errors are
  multiplicative, their sum is also a lower bound. 
  
  Let $\Gamma, \Gamma' \in \Graphp$ be respectively
  of diameters of order $(g+p)$ and $\log(g+p)$. 
  Consider a sequence $\Gamma_1 \ldots \Gamma_n$ of Whitehead moves
  taking $\Gamma$ to $\Gamma'$. Since 
  \[
  \Big| \diam(\Gamma_i) - \diam(\Gamma_{i+1}) \Big| \leq 1
  \]
   we must have
  \begin{equation} \label{Eq:p-Lower}
  \diam_W \big( \Graphp \big) \succ (g+p) - \log(g+p) \succ (g+p). 
  \end{equation} 
  That is the first lower bound. 
  
  When $g \ge p$,
  \eqnref{CardUnlabeled} reduces to 
  \begin{equation}
    \label{Eq:CardUnlabeled1}
    \big| \Graphp \big| \sim g^{g} \left( \frac{g}{p} \right)^p.
  \end{equation}
  On the other hand, when $p>g$, we have 
  \begin{equation}
    \label{Eq:CardUnlabeled2}
    \big| \Graphp \big| \sim p^{g} \left( \frac{p}{g} \right)^g.
  \end{equation}
  Let $r$ be the minimal number such that the ball of radius $r$ contains
  \Graphp. When $g \ge p$, \eqnref{CardUnlabeled1} implies \[
     r \succ g \log(g) + p \log \frac{g}{p} \succ g \log(g+p). 
  \]
  Similarly, when $p>g$, \eqnref{CardUnlabeled2} implies
  \[
    r \succ g \log(p) + g \log \frac{p}{g} \succ g \log(g+p). 
  \]
  Using the above and \eqnref{p-Lower} we obtain:
  \[ 
    \diam_W \big( \Graphp \big) \succ g\log(g+p)+ (g+ p).
  \] 
  This finishes the proof of \thmref{WHUnlabeled}.


  \bibliographystyle{amsalpha}
  \bibliography{main} 

\newcommand{\etalchar}[1]{$^{#1}$}
\providecommand{\bysame}{\leavevmode\hbox to3em{\hrulefill}\thinspace}
\providecommand{\MR}{\relax\ifhmode\unskip\space\fi MR }
\providecommand{\MRhref}[2]{%
  \href{http://www.ams.org/mathscinet-getitem?mr=#1}{#2}
}
\providecommand{\href}[2]{#2}
\begin{thebibliography}{BCG{\etalchar{+}}07}

\bibitem[BCG{\etalchar{+}}07]{bose:SDF}
P.~Bose, J.~Czyzowicz, Z.~Gao, P.~Morin, and D.~R. Wood, \emph{Simultaneous
  diagonal flips in plane triangulations}, J. Graph Theory \textbf{54} (2007),
  no.~4, 307--330.

\bibitem[Bol82]{boll:NRG}
B{\'e}la Bollob{\'a}s, \emph{The asymptotic number of unlabelled regular
  graphs}, J. London Math. Soc. (2) \textbf{26} (1982), no.~2, 201--206.

\bibitem[Bre09]{breslin:TT}
W.~Breslin, \emph{Thick triangulations of hyperbolic {$n$}-manifolds}, Pacific
  J. Math. \textbf{241} (2009), no.~2, 215--225.

\bibitem[Bus78]{buser:Eigenvalues}
Peter Buser, \emph{Cubic graphs and the first eigenvalue of a {R}iemann
  surface}, Math. Z. \textbf{162} (1978), no.~1, 87--99.

\bibitem[Bus92]{buser:GSC}
P.~Buser, \emph{Geometry and spectra of compact {R}iemann surfaces}, Progress
  in Mathematics, vol. 106, Birkh\"auser Boston Inc., Boston, MA, 1992.

\bibitem[Cav10]{cavendish:WH}
W.~Cavendish, \emph{Growth of the diameter of the pants graph modulo the
  mapping class group}, Preprint, {\tt
  https://web.math.princeton.edu/$\sim$wcavendi/PantsModMCG.pdf}, 2010.

\bibitem[CP10]{cavendish:WPD}
W.~Cavendish and H.~Parlier, \emph{Growth of the {W}eil-{P}etersson diameter of
  moduli space}, to appear in Duke Math. J. {\tt arXiv:math.GT/1004.3029},
  2010.

\bibitem[CV86]{vogtmann:OS}
M.~Culler and K.~Vogtmann, \emph{Moduli of graphs and automorphisms of free
  groups}, Invent. Math. \textbf{84} (1986), no.~1, 91--119.

\bibitem[FM10]{farb:MCG}
B.~Farb and D.~Margalit, \emph{A primer on mapping class groups}, Princeton
  Univ. Press, Princeton, N.J., 2010.

\bibitem[GL00]{gardiner:QT}
F.~P. Gardiner and N.~Lakic, \emph{Quasiconformal {T}eichm\"uller theory},
  Mathematical Surveys and Monographs, vol.~76, American Mathematical Society,
  Providence, RI, 2000.

\bibitem[Hub06]{hubbard:TT}
J.~Hubbard, \emph{{T}eichm\"uller theory and applications to geometry, topology
  and dynamics}, Matric Edition, Ithaca, NY, 2006.

\bibitem[Ker80]{kerckhoff:AG}
S.~P. Kerckhoff, \emph{The asymptotic geometry of {T}eichm\"uller space},
  Topology \textbf{19} (1980), no.~1, 23--41.

\bibitem[Lic64]{Lic}
W.~B.~R. Lickorish, \emph{A finite set of generators for the homeotopy group of
  a {$2$}-manifold}, Proc. Cambridge Philos. Soc. \textbf{60} (1964), 769--778.
  \MR{0171269 (30 \#1500)}

\bibitem[LL00]{leibon:DT}
G.~Leibon and D.~Letscher, \emph{Delaunay triangulations and {V}oronoi diagrams
  for {R}iemannian manifolds}, Proceedings of the {S}ixteenth {A}nnual
  {S}ymposium on {C}omputational {G}eometry ({H}ong {K}ong, 2000) (New York),
  ACM, 2000, pp.~341--349 (electronic).

\bibitem[LRT10]{tao:BC}
A.~Lenzhen, K.~Rafi, and J.~Tao, \emph{Bounded combinatorics and the
  {L}ipschitz metrc on {T}eichm\"uller space}, Preprint, {\tt
  arXiv:math.GT/1011.6078}, 2010.

\bibitem[Mar73]{margulis:ECE}
G.~A. Margulis, \emph{Explicit constructions of expanders}, Problemy Pereda\v
  ci Informacii \textbf{9} (1973), no.~4, 71--80.

\bibitem[Min96]{minsky:PR}
Y.~N. Minsky, \emph{Extremal length estimates and product regions in
  {T}eichm\"uller space}, Duke Math. J. \textbf{83} (1996), no.~2, 249--286.

\bibitem[PB10]{parlier:BPS}
H.~Parlier and F.~Balacheff, \emph{Bers' constants for punctured spheres and
  hyperelliptic surface}, Preprint, {\tt arXiv:math.GT/0911.5149}, 2010.

\bibitem[Pin73]{pinsker}
M.~S. Pinsker, \emph{On the complexity of a concentrator}, 7th International
  Telegraffic Conference (1973), 318/1--318/4.

\bibitem[STT88]{thurston:RDT}
D.~D. Sleator, R.~E. Tarjan, and W.~P. Thurston, \emph{Rotation distance,
  triangulations, and hyperbolic geometry}, J. Amer. Math. Soc. \textbf{1}
  (1988), no.~3, 647--681.

\bibitem[Thu86]{thurston:MSM}
W.~P. Thurston, \emph{{Minimal stretch maps between hyperbolic surfaces}},
  Preprint, {\tt arXiv:math.GT/9801039}, 1986.

\bibitem[Wol79]{wolpert:LS}
S.~A. Wolpert, \emph{The length spectra as moduli for compact {R}iemann
  surfaces}, Ann. of Math. (2) \textbf{109} (1979), no.~2, 323--351.

\end{thebibliography}

  \end{document}